\documentclass[11pt,a4paper]{article}

\usepackage{iftex}
\ifPDFTeX
\usepackage[T1]{fontenc}
\usepackage[utf8]{inputenc}
\fi
\usepackage[margin=1in]{geometry}
\usepackage{microtype}
\usepackage{indentfirst}
\usepackage{amssymb}
\usepackage{verbatim}
\usepackage{amsthm}
\usepackage{graphicx}
\ifPDFTeX
\usepackage{epstopdf}
\else
\ifLuaTeX
\usepackage{epstopdf}
\fi
\fi
\usepackage{amsfonts}
\usepackage{amsmath}
\usepackage{tikz}
\usepackage{bm}
\usepackage{pifont}
\usepackage{textcomp}
\usepackage{manyfoot}
\usepackage{booktabs}
\usepackage{appendix}

\usepackage{algorithm}
\usepackage{algorithmicx}
\usepackage{listings}
\usepackage{algpseudocode}
\usepackage[authoryear,round]{natbib}
\usepackage{xcolor}
\usepackage[authoryear,round]{natbib}
\usepackage[hidelinks]{hyperref}

\newcommand{\cE}{\mathcal{E}}

\newcommand{\cF}{\mathcal{F}}

\newcommand{\cK}{\mathcal{K}}

\newcommand{\cU}{\mathcal{U}}

\newcommand{\RR}{\mathbb{R}}

\newcommand{\PP}{\mathbb{P}}
\newcommand{\QQ}{\mathbb{Q}}
\newcommand{\drm}{\mathrm{d}}
\newcommand{\UU}{\mathbb{U}}
\newcommand{\EE}{\mathbb{E}}

\numberwithin{equation}{section}

\newcommand{\ep}{\epsilon}



\algrenewcommand\algorithmicrequire{\textbf{Input:}}
\algrenewcommand\algorithmicensure{\textbf{Output:}}
\allowdisplaybreaks

\newtheoremstyle{uniformstyle}
{12pt}        
{12pt}        
{\normalfont} 
{0pt}         
{\bfseries}   
{.}           
{.5em}        
{}            

\theoremstyle{uniformstyle}
\newtheorem{theorem}{Theorem}

\newtheorem{assumption}[theorem]{Assumption}
\newtheorem{lemma}[theorem]{Lemma}

\newtheorem{remark}{Remark}

\newif\ifdraftmode
\draftmodetrue
\ifdraftmode

\else

\fi

\title{\textbf{A Stochastic Maximum Principle for Partially Observed
		Jump-Diffusion Systems with State-Dependent Counting-Process Observations}}

\author{%
	Jie Xiong\thanks{Email: \href{mailto:xiongj@sustech.edu.cn}{xiongj@sustech.edu.cn}.}
	\qquad
	Ying Yang\thanks{Email: \href{mailto:12331007@mail.sustech.edu.cn}{12331007@mail.sustech.edu.cn}.}
	\\[0.75em]
	\small Department of Mathematics and SUSTech International Center for Mathematics\\
	\small Southern University of Science and Technology, Shenzhen 518055, P. R. China\\[0.35em] 
}

\date{}

\begin{document}
	
	\maketitle
	
	\begin{abstract}
		This paper studies a partially observed stochastic optimal control problem for
		jump-diffusion states observed through multivariate counting processes with
		state-dependent intensities. Because the observation jumps carry information
		about the latent state through their intensities, the variational analysis
		involves coupled state and likelihood-ratio perturbations. By introducing a
		reference probability measure and augmenting the state with the
		counting-process likelihood ratio, we establish a stochastic maximum principle
		whose first-order necessary condition is expressed as a conditional Hamiltonian
		stationarity relation with respect to the observation filtration. Finally, an
		LQ example is used to examine the compatibility of a linear adjoint ansatz and
		to illustrate numerically the coupled partial-information stationarity system.
	\end{abstract}
	
	\noindent\textbf{Keywords:} Stochastic maximum principle; partially observed
	control; jump-diffusion processes; counting-process observations; likelihood
	ratio; nonlinear filtering.
	
	\vspace{1.5em}

	\section{Introduction}\label{sec:intro}
	Partially observed stochastic control problems with discontinuous information arise when decisions are based on the arrival times and marks of random events rather than on diffusion-type observation paths. A standard mathematical formulation models the observation filtration as being generated by multivariate counting processes. When the counting intensities depend on the unobserved state, each observation jump carries information about the latent dynamics. This yields a coupled control problem, where the observation mechanism, nonlinear filtering dynamics, and variational analysis are inherently linked via state-dependent intensities.

	This observation structure is technically distinct from the classical Brownian setting. Under diffusion observations, the innovation process is a continuous martingale and the unnormalized filtering (Zakai) equation is driven by a Wiener integral. For counting-process observation models, observations arrive at random discrete times. The change of measure introduces a likelihood-ratio process $M$
	whose dynamics contain terms of the form $M(t-)(\lambda_k(X(t-),t-)-1) d(Y_k(t)-t)$.
	The resulting filter is driven by compensated counting martingales.
	Consequently, the variational analysis tracks both the perturbation
	of the latent state and the induced perturbation of the likelihood
	ratio.
	
The stochastic maximum principle (SMP) for systems with random jumps
has been extensively studied.  \cite{tang1994necessary}
derived necessary optimality conditions for fully observed stochastic
systems with random jumps, while  \cite{oksendal2010maximum} developed maximum principles for
forward--backward stochastic systems with jumps under partial
information. For partially observed systems with diffusion
observations, \cite{tang1998maximum} established a maximum
principle for controlled stochastic differential equations, and
 \cite{xiao2011maximum} extended this line of work to partially
observed forward--backward systems with random jumps. See also
\cite{bensoussan1992stochastic,wang2015linear,wang2008kalman,
	wang2013maximum} for related formulations.
	
A separated formulation uses the conditional law of the state as
	the information state; finite-dimensional filtering and
	certainty-equivalent control require additional structure. For a specific linear-quadratic model with state-dependent time-space point-process observations, \cite{snyder1977separation} obtained an exact optimal solution realized by a finite-dimensional nonlinear estimator and a certainty-equivalent linear controller. \cite{ceci1998partially} studied the equivalence between a partially observed control problem and its separated formulation for a pure Markov jump process observed through its total jump count. \cite{tang2002optimal} established a maximum principle for controlled point processes with Gaussian white-noise observations.   For models with counting-process observations, \cite{zeng2003partially} formulated a partially observed micromovement model for asset prices and developed a Markov chain approximation for the associated filtering equations; \cite{xiong2011branching} subsequently established convergence of a branching particle approximation to the optimal filter. As an application of counting-process observation models, \cite{xiong2020mean} derived a pre-commitment solution to a mean--variance portfolio selection problem.
	
The closest related maximum-principle result comes from
\cite{zheng2023global}. Their work addresses partially observed systems
with state-dependent marked-point-process intensities under mixed
Brownian--jump observations. Their observation process is formulated
as a semimartingale involving a Brownian component and a compensated
marked-jump integral, where the compensator of the underlying
integer-valued random measure depends on the unobserved state.
In contrast, the observation filtration in the present work is
generated directly by raw multivariate counting processes with
state-dependent intensities. Thus, the present observation structure
is not obtained simply by suppressing the drift term in their
observation equation.

In this paper, we study controlled jump-diffusion systems observed through multivariate
counting processes with state-dependent intensities
$\lambda_k(X(t-),t-)$. A control perturbation propagates through the
latent state to the observation intensities and likelihood process,
yielding a joint first-order state--likelihood variational system; this
coupling disappears when the intensities are state-independent. Under a
fixed reference measure, we establish the estimates needed for the
adjoint equations and derive a conditional Hamiltonian stationarity
condition with respect to the observation filtration.

The contributions are as follows.
	\begin{enumerate}
		\item We formulate the joint controlled state--likelihood
		system under a reference measure, derive its first-order
		variational equations, and establish the required moment estimates.
		
		\item We establish a stochastic maximum principle giving a
			conditional Hamiltonian stationarity condition with respect to the
			observation filtration, and establish the well-posedness of the
			associated adjoint equations in explicitly specified solution spaces.
	\item For a trading-motivated LQ model, we derive the compatibility
		condition for a deterministic linear adjoint ansatz and	numerically examine the observation-induced coupling in the	necessary optimality system.
	\end{enumerate}

	The remainder of the paper is organized as follows. Section~\ref{sec:formulation} formulates the control problem and introduces the reference probability measure. Section~\ref{sec:SMP} establishes the stochastic maximum principle. Section~\ref{sec:app} develops the LQ illustration, examines the linear adjoint ansatz, and presents a numerical approximation of the coupled partial-information adjoint system. Section~\ref{sec:conclusion} concludes. Proofs of technical lemmas are collected in Appendix~\ref{App:1}, implementation details are provided in Appendix~\ref{App:2}, and numerical-resolution sensitivity results are reported in Appendix~\ref{App:3}.
	
	\section{Problem formulation and preliminaries}\label{sec:formulation}
	\subsection{Problem formulation}
	The one-dimensional formulation is adopted to keep the notation
	transparent. Extensions to multidimensional state processes can be
	developed under the corresponding vector and matrix notation, but are not pursued here.
	
	Let $T>0$ be fixed. Consider a complete filtered probability space $(\Omega, \mathcal{F}, \{\mathcal{F}_t\}_{t \in [0,T]}, \mathbb{P})$ supporting a standard Brownian motion $W$ and an independent Poisson random measure $\bar{N}(\mathrm{d}\zeta, \mathrm{d}t)$ on $\mathbb{R}^+ \times \mathcal{E}$ with a $\sigma$-finite intensity measure $\nu(\mathrm{d}\zeta)\mathrm{d}t$. Here, $(\mathcal{E}, \mathcal{B}(\mathcal{E}))$ is a Polish space satisfying $\int_{\mathcal{E}} (1 \wedge |\zeta|^2) \nu(\mathrm{d}\zeta) < \infty$.  We denote by $\tilde{N} = \bar{N} - \nu \mathrm{d}t$ the associated compensated random measure. Throughout this paper, $\mathbb{E}^{\mathbb{P}}$ denotes the expectation taken with respect to the probability measure $\mathbb{P}$. Consider the following stochastic differential equation (SDE) with jumps governing the asset price $X(t)$ on $[0,T]$:
	\begin{equation}\label{eq:state}
		\left\{
		\begin{aligned}
			\drm X(t)&=b\left(t,X(t),u(t)\right)\drm t+\sigma\left(t,X(t),u(t)\right) \drm W(t)\\
			&\quad+\int_\cE \eta(t,X(t-),u(t),\zeta)\tilde{N}(\drm \zeta,\drm t), \\
			X(0)&=x\in\RR,
		\end{aligned}
		\right.
	\end{equation}
	where $u(\cdot)$ is a control process taking values in a non-empty convex set \(\mathbb U\subseteq\mathbb R\) with \(0\in\mathbb U\), and $W$ is a Brownian motion that is independent of \(\tilde{N}\). Moreover, $b:[0,T]\times \RR\times \UU\to \RR$,  $\sigma:[0,T]\times\RR\times\UU\to \RR$, \(\eta:[0,T]\times\mathbb R\times U\times \cE\to\mathbb R\).
	
	Let $y_1,\ldots,y_n$ denote the modeled transaction-price levels.
	The observed trade arrivals are represented by the multivariate
	counting process $\overrightarrow Y=(Y_1,\ldots,Y_n)$, where
	$Y_k(t)$ counts trades observed at price level $y_k$ up to time $t$. Therefore, following \cite{zeng2003partially} and \cite{xiong2011branching}, we model the prices as a collection of counting processes with the following specification
	\begin{equation}\label{eq:observe}
		\overrightarrow{Y}(t)=
		\begin{pmatrix}
			N_1(\int_0^t \lambda_1( X(s-),s-)ds) \\
			N_2(\int_0^t \lambda_2( X(s-),s-)ds)\\
			\cdots \\
			N_n(\int_0^t \lambda_n(X(s-),s-)ds)
		\end{pmatrix},
	\end{equation} 
	where its $k$-th component, $Y_k(t)$, counts the number of trades executed at the $k$-th price level up to time $t$. Consequently, the pair $(X, \overrightarrow{Y})$ constitutes a partially observed system driven by counting-process observations.

	We want to study the optimal control problem associated with the cost functional
	\begin{equation}\label{eq:cost}
		\begin{aligned}
			J(u) = \EE^\PP[\Phi(X(T)) + \int_0^T f(t,X(t),u(t)) dt],
		\end{aligned}
	\end{equation}
	where $f:[0,T]\times\mathbb{R}\times\UU\to \mathbb{R}$, and $\Phi:\mathbb{R}\rightarrow \mathbb{R}$ are deterministic functions.  
	
	Let $\mathcal{F}_t^{\overrightarrow{Y}} = \sigma\{ \overrightarrow{Y}(s) \mid 0 \leq s \leq t\}$ denote the filtration generated by the observation process, representing the information available up to time $t$. A control process $u(\cdot)$ is called admissible if it is a $\mathbb U$-valued, $\mathbb F^{\overrightarrow Y}$-predictable process satisfying $\mathbb E^{\mathbb P}\int_0^T |u(t)|^{4p_0}\,\mathrm dt<\infty$
for some $p_0>1$. We write $\mathcal U_{ad}$ for the collection of all admissible controls.
	The optimal control problem considered in this paper is to find an optimal control $\tilde{u}(\cdot) \in \mathcal{U}_{ad}$ that minimizes the cost functional \eqref{eq:cost}, that is,$$J(\tilde{u}(\cdot)) = \inf_{u(\cdot) \in \mathcal{U}_{ad}} J(u(\cdot)).$$
	
	For $1 \leq p < \infty$, let $L^p(\cE, \nu)$ denote the space of $B(\cE)$-measurable functions $g \colon \cE \to \mathbb{R}$ modulo $\nu$-a.e.\ equality such that $\|g\|_{L^p(\cE,\nu)} := \bigl(\int_\cE |g(\zeta)|^p \nu(\drm \zeta)\bigr)^{1/p} < \infty$. 
	To establish the stochastic maximum principle for this partially observed problem, we impose the following assumptions.  
	\begin{assumption}\label{Ass:1}
		\begin{enumerate}
			\item $b(\cdot, 0, 0), \sigma(\cdot, 0, 0)$,   $\int_{\mathcal{E}} |\eta(\cdot, 0, 0,  \zeta)|^2 \nu(\drm\zeta)$, $f(\cdot,0,0)$, and $\Phi(0)$ are uniformly bounded.
			\item The coefficients $b$ and $\sigma$ are continuously differentiable with respect to $(x,u)$, and are bounded by $k(1+|x|+|u|)$ for some constant $k>0$. Furthermore, their derivatives $b_x, b_u, \sigma_x$, and $\sigma_u$ are continuous and uniformly bounded.
			\item For every \(t\in[0,T]\), the mapping $(x,u)\mapsto\eta(t,x,u,\cdot)$ 
			is continuously differentiable from \(\RR\times\UU\) into
			\(L^2(\cE,\nu)\cap L^4(\cE,\nu)\). And, there exists a constant $k_{\eta}>0$ such that for all \((t,x,u)\in[0,T]\times\RR\times\UU\), $\int_{\mathcal{E}}|\eta(t,x,u,\zeta)|^2\nu(\drm\zeta) \le k_{\eta}(1+|x|^2+|u|^2)$, and $\int_{\mathcal{E}}|\eta(t,x,u,\zeta)|^4\nu(\drm\zeta) \le k_{\eta}(1+|x|^4+|u|^4)$. Furthermore, there exists a constant \(L>0\) such that for all \((t,x,u)\in[0,T]\times\RR\times\UU\),
				$$ \int_{\cE}\left(|\eta_x(t,x,u,\zeta)|^2+|\eta_u(t,x,u,\zeta)|^2+|\eta_x(t,x,u,\zeta)|^4+|\eta_u(t,x,u,\zeta)|^4\right)\nu(\drm \zeta)\le L.$$

			\item The cost functions $f$ and $\Phi$ are continuously differentiable with respect to $(x,u)$ and $x$, respectively. Moreover, there exists a constant $k_c>0$ such that $|f_x(t,x,u)|+|f_u(t,x,u)| \le k_c(1+|x|+|u|)$, and $|\Phi_x(x)| \le k_c(1+|x|)$.
		\end{enumerate}
	\end{assumption} 
	\begin{remark}
		Under Assumption~\ref{Ass:1}, equation~\eqref{eq:state} is well posed
		by standard arguments. Moreover, the \(L^2(\cE,\nu)\cap
			L^4(\cE,\nu)\)-valued \(C^1\)-regularity in item~(3) implies that,
			for every \(t\in[0,T]\) and \(p\in\{2,4\}\), the mappings $
			(x,u)\mapsto\eta_x(t,x,u,\cdot)$ and $  (x,u)\mapsto\eta_u(t,x,u,\cdot)$  are continuous from \(\RR\times\UU\) into \(L^p(\cE,\nu)\).
	\end{remark} 
	\begin{assumption}\label{Ass:2}
			Under \(\mathbb P\), \(N_1,\ldots,N_n\) are independent Poisson processes with unit intensity. Moreover, they are independent of the Brownian motion $W$ and the Poisson random measure $\bar N$.
	\end{assumption}
	Since $u$ is $\mathcal F^{\vec Y}$-predictable, the control at time $t>0$ may depend on the observation history strictly before $t$. Through the control, past observations may affect the subsequent state dynamics. Thus, the independence of the primitive noises in Assumption~\ref{Ass:2} does not imply
		independence between the controlled state and the primitive Poisson clocks.

	\begin{assumption}\label{Ass:3}
		The intensity at level $k$ takes the form
		$\lambda_k(x,t)=a(x,t) p_k(x)$, where $a(x,t)=\sum_{k=1}^{n}\lambda_k(x,t)$
		is the total intensity at time $t$ with state $x=X(t)$, and
		$p_k(x) = p(y_k\mid x)$ is the transition probability from $x$
		to the $k$th price level $y_k$, satisfying
		$\sum_{k=1}^{n}p_k(x) = 1$ for all $x\in\RR$.
	\end{assumption}

	\begin{assumption}\label{Ass:4} 
		\begin{enumerate}
			\item The total intensity $a(x,t)$ is uniformly bounded from above and  below, that is, there exist constants $0<K_1<K_2$ such that  $ K_1 \leq a(x,t) \leq K_2,$  $\forall  (x,t)\in\RR\times[0,T].$ Furthermore, $a(\cdot,t)$ is continuously differentiable in $x$, and $a'(x,t)=\frac{\drm}{\drm x}a(x,t)$ is continuous in $x$ and uniformly  bounded in $(x,t)\in\RR\times[0,T]$.
			\item For each $k$, $p_k$ is continuously differentiable with $p_k'$ continuous and bounded. Moreover, $
			\delta_p :=\inf_{x\in\RR, 1\le k\le n} p_k(x) > 0.$
		\end{enumerate}
		
	\end{assumption}
	\begin{remark}
		Assumptions~\ref{Ass:3} and~\ref{Ass:4} imply $0<\delta_\lambda:=K_1\delta_p\leq\lambda_k(x,t)\leq K_2$ 
		for all $k$, $x$, and $t$. Hence, $\log\lambda_k$ is uniformly bounded, as required for the moment estimates in Lemma~\ref{lem:moment-PQ}.
		
		The lower bound $K_1>0$ imposes a uniformly positive total trading intensity. The condition $\delta_p>0$ further requires every modeled price level to retain a uniformly positive conditional probability at a trade, for every latent state.
	\end{remark}
	
	\subsection{Reference-measure formulation}\label{sub:co}
	Let $\{\hat{\mathcal F}_t\}_{t\in[0,T]}$ be the usual augmentation of the joint natural filtration generated by $W$, $\tilde N$, and $\overrightarrow Y$. We work on the filtered probability space $(\Omega,\hat{\mathcal F}_T, \{\hat{\mathcal F}_t\}_{t\in[0,T]},\mathbb P)$. Under Assumptions~\ref{Ass:2}--\ref{Ass:4}, the Girsanov theorem for point processes implies that there exists a reference probability measure $\mathbb{Q}$ that is equivalent to $\mathbb{P}$ on $\hat{\mathcal{F}}_T$.  Under $\mathbb Q$,  $W$ is a standard Brownian motion, $\bar N$ is a Poisson random measure with compensator $\nu(\mathrm d\zeta)\mathrm dt$, and $Y_1,\ldots,Y_n$ are Poisson processes with unit intensity. These reference drivers are mutually independent.

	The measure change is governed by the Radon--Nikodym derivative $
	M(t)=\left.\frac{\mathrm{d}\mathbb{P}}{\mathrm{d}\mathbb{Q}}\right|_{\hat{\mathcal{F}}_t}$,
	which is a uniformly integrable $\mathbb{Q}$-martingale representing the counting-process likelihood ratio (see, e.g., \cite{zeng2003partially}), and satisfies
	\begin{align}\label{eq:M}
		\mathrm{d}M(t)&=\sum_{k=1}^n\bigl(\lambda_k(X(t-),t-)-1\bigr)M(t-)\mathrm{d}\bigl(Y_k(t)-t\bigr),
	\end{align}
	with $M(0)=1$.
	
	Under $\mathbb Q$, the likelihood-ratio martingale is orthogonal to $W$ and to stochastic integrals with respect to $\widetilde N$. Hence the Brownian and state-jump characteristics in \eqref{eq:state} are unchanged under the measure transformation. In particular, for
		$g\in C_b(\mathbb R)$, $\mathbb E^{\mathbb P}[g(X(t))] = \mathbb E^{\mathbb Q}[M(t)g(X(t))].$  We define the unnormalized filter by $\langle V(t),g\rangle:=\mathbb E^{\mathbb Q} \left[M(t)g(X(t))\mid\mathcal F_t^{\overrightarrow Y}\right]$ for $g\in C_b(\mathbb R)$. Let $\Pi(t)$ denote the conditional distribution of $X(t)$ given $\mathcal F_t^{\overrightarrow Y}$ under $\mathbb P$. Then the Kallianpur--Striebel formula gives
	\begin{equation}\label{eq:ksfor}
		\langle\Pi(t),g\rangle
		= \mathbb E^{\mathbb P} \bigl[g(X(t))\mid\mathcal F_t^{\overrightarrow Y}\bigr]
		= \frac{\langle V(t),g\rangle}{\langle V(t),1\rangle}.
	\end{equation}

	Under $\QQ$, the augmented state equation is given by
	\begin{equation}\label{eq:Qstate}
		\left\{
		\begin{aligned}
			\drm X(t)&=b(t,X(t),u(t))\drm t+\sigma(t,X(t),u(t))\drm W(t)\\
			&\quad+\int_\cE\eta(t,X(t-),u(t),\zeta)\tilde N(\drm\zeta,\drm t),\\
			\drm M(t)&=\sum_{k=1}^{n}\bigl(\lambda_k(X(t-),t-)-1\bigr)M(t-)\drm\bigl(Y_k(t)-t\bigr),\\
			X(0)&=x\in\RR,\qquad M(0)=1.
		\end{aligned}
		\right.
	\end{equation}

	Under the reference measure $\QQ$, the cost functional \eqref{eq:cost} can be rewritten as
	\begin{align}\label{eq:Qcost}
		J(u)=\EE^\QQ\left[\Phi(X(T))M(T)+\int_0^T M(t)f(t,X(t),u(t)) dt\right].
	\end{align}
	Thus, the original optimization problem is equivalent to minimizing the cost functional \eqref{eq:Qcost} over all admissible controls $u(\cdot)\in\mathcal U_{ad}$, subject to the augmented state equation \eqref{eq:Qstate}.

	\section{Stochastic maximum principle}\label{sec:SMP}
	In this section, we derive the stochastic maximum principle for the problem of minimizing \eqref{eq:Qcost} subject to \eqref{eq:Qstate}. Let $u(\cdot)\in\mathcal U_{ad}$ be an optimal control, and let $(X(\cdot),M(\cdot))$ be the corresponding optimal state trajectory. We impose the following two-sided variation condition on $u(\cdot)$:
	for every bounded, $\mathbb R$-valued, $\mathbb F^{\overrightarrow Y}$-predictable process $v(\cdot)$, there exists $\epsilon_0>0$ such that $
	u^\epsilon(\cdot):=u(\cdot)+\epsilon v(\cdot)\in\mathcal U_{ad}$, for $|\epsilon|<\epsilon_0.$ Let $(X^\epsilon(\cdot),M^\epsilon(\cdot))$ denote the state process corresponding to $u^\epsilon(\cdot)$.

		All controlled systems are realized on the same reference space under the fixed probability measure $\mathbb Q$, with common drivers $W,\bar N$, and $\vec Y$. Consequently, $(X,M)$ and $(X^\epsilon,M^\epsilon)$ are driven by the same coordinate observation process, and the observation filtration $\mathbb F^{\vec Y}$ remains fixed throughout the variational analysis. The dependence of the associated physical measure on the control is suppressed in the notation.

	For notational simplicity, we write $b(t):=b(t,X(t),u(t))$, $\sigma(t):=\sigma(t,X(t),u(t))$,
	$f(t):=f(t,X(t),u(t))$, $\eta(t,\zeta):=\eta(t,X(t-),u(t),\zeta)$, and $\lambda_k(t):=\lambda_k(X(t),t)$.
	The same convention applies to the perturbed pair $(X^\epsilon,u^\epsilon)$. In particular,
	$b^\epsilon(t):=b(t,X^\epsilon(t),u^\epsilon(t))$ and $\lambda_k^\epsilon(t):=\lambda_k(X^\epsilon(t),t)$,
	with $\sigma^\epsilon$, $\eta^\epsilon$, and $f^\epsilon$ defined analogously.

	We first establish moment estimates for the likelihood ratio $M$, 
	which ensure the required integrability under~$\QQ$. The proof uses  standard moment estimation techniques and is deferred to Appendix \ref{App:1}.
	\begin{lemma}\label{lem:moment-PQ}
		Under Assumptions \ref{Ass:3} and \ref{Ass:4}, for any $q\in\RR$,
		\begin{align}\label{eq:Mq}
			\sup_{0\le t\le T}\EE^\QQ[|M(t)|^q] < \infty.
		\end{align}
		Consequently, for any $u\in\cU_{ad}$,
		\begin{align}\label{eq:up}
			\EE^\QQ \int_0^T |u(t)|^4 \drm t < \infty.
		\end{align}
	\end{lemma}
	
	The following lemma establishes the finiteness of the fourth moment of the state process $X$. The detailed proof is deferred to the Appendix \ref{App:1} for brevity.
	\begin{lemma}\label{lem:M4}
		Let Assumptions \ref{Ass:1} to \ref{Ass:4} hold. For any \(u(\cdot)\in\cU_{ad}\), there exists a constant \(C>0\) such that the solution \(X(\cdot)\) to the state equation \eqref{eq:Qstate} satisfies
		\begin{align}
			&\sup_{0\leq t\leq T}\EE^\QQ \lvert X(t)\rvert^4\leq  C\left(1+ \EE^\QQ  \int_0^T \lvert u(s)\rvert^4 \drm s\right) .\label{eq:X8}
		\end{align} 
	\end{lemma}

	We next establish moment estimates for the variational processes associated with the state trajectory.
	\begin{lemma}\label{lem:X-X}
		Let Assumptions \ref{Ass:1} - \ref{Ass:4} hold. Then,  there exist a constant $C$ such that 
		\begin{align}
			\sup_{0\leq t\leq T}\EE^\QQ\lvert X^{\epsilon}(t)-X(t)\rvert^4\leq C \epsilon^4,\label{eq:X4}\\
			\sup_{0\leq t\leq T}\EE^\QQ \lvert M^{\epsilon}(t)-M(t)\rvert^2 \leq C \epsilon^2.\label{eq:M4}
		\end{align}
	\end{lemma}
	\begin{proof}
		For simplicity, we denote $\xi^\epsilon(t)=X^\epsilon(t)-X(t).$
		Note that  $
		\drm \xi^\epsilon(t)=\left[b^\epsilon(t)-b(t)\right]\drm t+\left[ \sigma^\epsilon(t)-\sigma(t)\right]\drm W(t)+\int_{\mathcal{E}}\left[\eta^\epsilon(t ,\zeta)-\eta(t,\zeta)\right]\tilde{N}(\drm \zeta,\drm t),$
		and $\xi^\epsilon(0)=0.$  Then, by the B-D-G inequality, the Kunita inequality, the Young's inequality, and the Cauchy-Schwarz  inequality,  we have that 
		\begin{equation}\label{eq:esxi8}
			\begin{split}
				&\EE^\QQ \lvert \xi^\epsilon(t) \rvert^4\\
				&=\EE^\QQ \bigg\{\Big\lvert\int_0^t \left[b^\ep(s )-b(s )\right]\drm s+\int_0^t\left[ \sigma^\ep(s )-\sigma(s )\right]\drm W(s)+\int_0^t\int_{\mathcal{E}}\left[\eta^\ep(s ,\zeta)-\eta(s,\zeta)\right]\tilde{N}(\drm \zeta,\drm s)\Big\rvert^4\bigg\}\\
				&\leq C^1 \EE^\QQ\bigg\{ \int_0^t\left\lvert b_x^\rho(s)\xi^\epsilon(s)+b_u^\rho(s)\epsilon v(s)\right\rvert^4\drm s+  \int_0^t \left\lvert \sigma_x^\rho(s)\xi^\epsilon(s)+\sigma_u^\rho(s)\epsilon v(s)\right\rvert^4 \drm s\bigg\}\\
				&\quad+C^1\EE^\QQ\left\lvert\int_0^t\int_{\cE}\left\lvert\eta_x^\rho(s,\zeta)\xi^\epsilon(s-)+\eta_u^\rho(s,\zeta)\epsilon v(s)\right\rvert^2\nu(\drm \zeta)\drm s\right\rvert^2\\
				&\quad+27C_D \EE^\QQ\int_0^t \int_{\cE}\left\lvert\eta_x^\rho(s,\zeta)\xi^\epsilon(s-)+\eta_u^\rho(s,\zeta)\epsilon v(s)\right\rvert^4 \nu(\drm \zeta)\drm s,
			\end{split}	
		\end{equation}	 
		with 
			\begin{equation}\label{eq:mth}
				\begin{aligned}
					b_x^\rho(s)&=\int_0^1 b_x(s,X(s)+\rho\xi^\ep(s),u(s)+\rho\epsilon v(s))\drm \rho,\\
					b^\rho_u(s)&=\int_0^1b_u(s,X(s)+\rho\xi^\epsilon(s),u(s)+\rho\epsilon v(s))\drm\rho,\\
					\sigma_x^\rho(s)&=\int_0^1\sigma_x\bigl(s,X(s)+\rho\xi^\epsilon(s),u(s)+\rho\epsilon v(s)\bigr)\drm \rho,\\
					\sigma^\rho_u(s)&=\int_0^1\sigma_u(s,X(s)+\rho\xi^\epsilon(s),u(s)+\rho\epsilon v(s))\drm\rho, \\
					\eta_x^\rho(s,\zeta)&=\int_0^1 \eta_x(s,X(s-)+\rho\xi^\ep(s-),u(s)+\rho\epsilon v(s),\zeta)\drm \rho,\\
					\eta^\rho_u(s,\zeta)&=\int_0^1\eta_u(s,X(s-)+\rho\xi^\epsilon(s-),u(s)+\rho\epsilon v(s),\zeta)\drm\rho,
				\end{aligned}
			\end{equation} 
			and \(C^1=\max\{27T^3,27C_G^1T ,27C_D^1\}\). It follows that all averaged derivatives featured in \eqref{eq:mth} are evaluated along the common line segment spanning $(X,u)$ and $(X^\epsilon,u^\epsilon)$. The integral‑form Taylor expansions $b^\epsilon-b=b_x^\rho\xi^\epsilon+b_u^\rho\epsilon v$, $\sigma^\epsilon-\sigma=\sigma_x^\rho\xi^\epsilon+\sigma_u^\rho\epsilon v$, and $
			\eta^\epsilon(s,\zeta)-\eta(s,\zeta)=\eta_x^\rho(s,\zeta)\xi^\epsilon(s-)
			+\eta_u^\rho(s,\zeta)\epsilon v(s)$, therefore hold as exact equalities.

		Since $b_x$, $b_u$, $\sigma_x$, and $\sigma_u$ are uniformly
		bounded by Assumption~\ref{Ass:1}, applying the elementary inequality $(a+b)^4 \le 2^3(a^4+b^4)$ yields $ \EE^\QQ\Big\{ \int_0^t\left\lvert b_x^\rho(s)\xi^\epsilon(s)+b_u^\rho(s)\epsilon v(s)\right\rvert^4\drm s+  \int_0^t \left\lvert \sigma_x^\rho(s)\xi^\epsilon(s)+\sigma_u^\rho(s)\epsilon v(s)\right\rvert^4 \drm s\Big\} \le C_2^1 \EE^\QQ \int_0^t|\xi^\epsilon(s)|^4\drm s + C_2^1\epsilon^4\EE^\QQ\int_0^t|v(s)|^4\drm s$,
		where $C_2^1$ is a positive constant depending on $C_1$ and the uniform bounds of $b_x^\rho, b_u^\rho, \sigma_x^\rho, \sigma_u^\rho $.
		
		Moreover,   based on Assumption \ref{Ass:1},   we have that  $$
		\begin{aligned} &C^1\EE^\QQ\left\lvert\int_0^t\int_{\cE}\left\lvert\eta_x^\rho(s,\zeta)\xi^\epsilon(s-)+\eta_u^\rho(s,\zeta)\epsilon v(s)\right\rvert^2\nu(\drm \zeta)\drm s\right\rvert^2 \\
			&+27C_D \EE^\QQ\int_0^t \int_{\cE}\left\lvert\eta_x^\rho(s,\zeta)\xi^\epsilon(s-)+\eta_u^\rho(s,\zeta)\epsilon v(s)\right\rvert^4 \nu(\drm \zeta)\drm s \\
			&\le  C_2^2 \EE^\QQ \int_0^t|\xi^\epsilon(s-)|^4\drm s + C_2^2\epsilon^4\EE^\QQ\int_0^t|v(s)|^4\drm s,
		\end{aligned}$$
		where $C_2^2$ is a positive constant depending on $C_1$ and $L$ given in Assumption \ref{Ass:1}.
		
		Since $\xi^\epsilon$ is c\`adl\`ag, $\xi^\epsilon(s-)=\xi^\epsilon(s)$ for
			$\mathrm ds$-a.e. $s\in(0,T]$. Thus, $\int_0^t |\xi^\epsilon(s-)|^p \mathrm ds =\int_0^t |\xi^\epsilon(s)|^p \mathrm ds$, $\mathbb Q$ -a.s. Therefore, we have  $
		\EE^\QQ|\xi^\ep(t)|^4\le C_2 \EE^\QQ \int_0^t|\xi^\epsilon(s)|^4\drm s + C_2 \epsilon^4\EE^\QQ\int_0^t|v(s)|^4\drm s.$
		By using Gronwall's inequality  to $\EE^\QQ|\xi^\epsilon(t)|^4$, we obtain $
		\EE^\QQ|\xi^\epsilon(t)|^4 \le C_2\epsilon^4 \Big(\EE^\QQ \int_0^T |v(s)|^4\drm s\Big) e^{C_2T}$. 
		
		Since $v(\cdot)$ is bounded, $\EE^\QQ\int_0^T|v(s)|^4\drm s
		<\infty$, so there exists a constant $\kappa$ such that 
		\begin{align}\label{eq:xi8}
			\EE^\QQ |\xi^\epsilon(t)|^4 \leq \kappa \epsilon^4,
		\end{align}
		which completes the proof of inequality \eqref{eq:X4}.
		
		We now prove inequality \eqref{eq:M4}. For notational simplicity, define $\bar{M}^\epsilon(t) = M^\epsilon(t) - M(t)$. Then,
		$\drm \bar{M}^\epsilon(t) =\sum_{k=1}^n \Big[ \lambda_k^\epsilon(t-)M^\epsilon(t-)-\lambda_k(t-)M(t-)+M(t-) -M^\epsilon(t-)\Big]\drm ( Y_k(t)-t),$
		and $\bar{M}_0^\epsilon=0$.  
		
		For the first term, by applying the Cauchy--Schwarz inequality, we have $\EE^\QQ[|\xi^\epsilon|^2|M^\epsilon|^2] 
		\le\big(\EE^\QQ|\xi^\epsilon|^4\big)^{1/2}\big(\EE^\QQ|M^\epsilon|^4\big)^{1/2}$.    By Lemma \ref{lem:moment-PQ} applied with the perturbed control $u^\epsilon$ in place of $u$, we have  $
		\sup_{0<|\epsilon|<\epsilon_0}\sup_{0\le t\le T}\EE^\QQ|M^\epsilon(t)|^4 < \infty. $
		Therefore, combining with \eqref{eq:xi8}, $ \EE^\QQ\int_0^t |\xi^\epsilon(s-)|^2|M^\epsilon(s-)|^2\drm s \le \int_0^t (\EE^\QQ|\xi^\epsilon(s)|^4)^{1/2}(\EE^\QQ|M^\epsilon(s)|^4)^{1/2}\drm s \le C^4 \epsilon^2$.  Here $C^4$ is a positive constant which is independent of $t$.
		
		Substituting into the estimate for $\EE^\QQ|\bar M^\epsilon|^2$, we have
		$
		\EE^\QQ|\bar M^\epsilon(t)|^2 \le C \epsilon^2 + C \int_0^t \EE^\QQ|\bar M^\epsilon(s)|^2\drm s.
		$
		By using Gronwall's inequality, there exists a constant $C>0$ such that
		$\sup_{0\le t\le T}\EE^\QQ|\bar M^\epsilon(t)|^2 
		\le C\epsilon^2$,
		which establishes \eqref{eq:M4}.
	\end{proof}

We now consider the following linear variational equations
		for $(X^l,M^l)$:
	\begin{equation}\label{eq:Psi}
		\left\{\begin{aligned}
			\drm X^l(t)=&\left[b_x X^l +b_u v \right]\drm t+\left[\sigma_x X^l +\sigma_u v \right]\drm W(t)\\
			&+\int_{\cE}\left[\eta_x(t,\zeta)X^l(t-)+\eta_u(t,\zeta)v(t)\right]\tilde{N}(\drm \zeta,\drm t),\\
			X^l(0)=&0,
		\end{aligned}
		\right.
	\end{equation}
	and
	\begin{equation}\label{eq:varphi}
		\left\{\begin{aligned}
			\drm M^l(t)=&\sum_{k=1}^n\big\{\left(\lambda_k(t-)-1\right)M^l(t-) + \lambda_k^\prime(t-)M(t-)X^l(t-)\big\}\drm \left( Y_k(t)-t\right),\\
			M^l(0)=&0.
		\end{aligned}
		\right.
	\end{equation}
Under Assumptions~\ref{Ass:1} to \ref{Ass:4},
		the linear system given by \eqref{eq:Psi} and \eqref{eq:varphi}
		admits a unique adapted solution $(X^l,M^l)$.
	
	Now, we present the following lemma to establish the estimation of $X^l$ and $M^l$,the proof of which is detailed in Appendix \ref{App:1}.
		\begin{lemma}\label{lem:XL}
			Under Assumptions \ref{Ass:1} to \ref{Ass:4}, the solution $(X^l, M^l)$ of the linear SDE system \eqref{eq:Psi} and \eqref{eq:varphi} respectively satisfies 
			\[\sup_{0\le t\le T}\EE^\QQ|X^l(t)|^4 < \infty, \qquad 
			\sup_{0\le t\le T}\EE^\QQ|M^l(t)|^2 < \infty.
			\]
		\end{lemma}

	To derive the stochastic maximum principle, for \(t\in[0,T]\) and \(\epsilon>0\), we define  
	\[
	\tilde{\varrho}(t) = \frac{\varrho^\epsilon(t) - \varrho(t)}{\epsilon} - \varrho^l(t), \quad \text{for} \quad \varrho = X, M.
	\]
	The following lemma establishes the convergence of $\tilde{\varrho}(t)$ for both $\varrho = X$ and $\varrho = M$. A detailed proof is provided in Appendix \ref{App:1}.
	\begin{lemma}\label{lem:LXM}
		Let Assumptions \ref{Ass:1} to \ref{Ass:4} hold. Then,   
		\begin{align}
			&\lim_{\epsilon\to 0}\sup_{0\leq t\leq T} \EE^\QQ\lvert \tilde{X}(t)\rvert^4 =0,\label{eq:X0}\\
			&	\lim_{\epsilon\to 0}\sup_{0\leq t\leq T}  \EE^\QQ\lvert \tilde{M}(t)\rvert^2 =0.\label{eq:M0}
		\end{align}
	\end{lemma}

		Thus, \(X^l\) and \(M^l\) are the first-order variations of \(X\) and \(M\) in the respective moment norms.
		
		Defining $J(\epsilon):=J(u+\epsilon v)$, the directional derivative of $J$ at $u$ along $v$ is $\frac{d}{d\epsilon}J(u+\epsilon v)|_{\epsilon=0}:=\lim_{\epsilon\to 0}\epsilon^{-1}[J(\epsilon)-J(0)]$. Since $u$ is an interior optimal control, $u\pm\epsilon v$ remain admissible for sufficiently small $\epsilon>0$. Thus, $\epsilon=0$ locally minimizes $J(\epsilon)$, implying $\frac{d}{d\epsilon}J(u+\epsilon v)|_{\epsilon=0}=0$ provided the limit exists.
	
	Substituting the expansions $X^\epsilon = X + \epsilon X^l + \epsilon\tilde{X}$ and $M^\epsilon = M + \epsilon M^l + \epsilon\tilde{M}$ into the cost functional, we expand $\Phi(X^\epsilon)M^\epsilon$ and $f(t,X^\epsilon,u^\epsilon)M^\epsilon$ to first order. By applying  Lemmas~\ref{lem:X-X} and~\ref{lem:LXM},  we obtain  
	\begin{equation}\label{eq:FD}
		\begin{aligned}
			0=&   \frac{\drm}{\drm \ep} J(\ep)\left|_{\epsilon=0}\right.=\lim_{\ep\to 0}\frac{J(u(\cdot)+\ep v(\cdot))-J(u(\cdot))}{\ep}\\
			=&\EE^\QQ\left[\Phi(X(T))M^l(T)+\Phi^\prime(X(T))M(T)X^l(T)\right] +\EE^\QQ\Big[\int_0^T f(t,X(t),u(t))M^l(t)\\
			&+M(t)\left(f_x\left(t,X(t),u(t)\right)X^l(t)+f_u\left(t,X(t),u(t)\right)v(t)\right)\drm t\Big].
		\end{aligned}
	\end{equation}
	All terms in \eqref{eq:FD} are integrable by H\"older's inequality, Assumption~\ref{Ass:1}, and Lemmas~\ref{lem:moment-PQ}, \ref{lem:M4}, and~\ref{lem:XL}.
	
	We define the Hamiltonian function \(H\) by $
	H(t,u,x,m,\alpha,\beta,\gamma^1,\cdots,\gamma^n,\theta)= b(x,u)\alpha+
	\sigma(x,u)\beta+\sum_{k=1}^n \gamma^k\left(\lambda_k(x,t)-1\right)m 
	+\int_{\mathcal{E}}\eta(t,x,u,\zeta)\theta(t,\zeta)\nu(\drm \zeta)+m f(t,x,u).$

	Then, we formulate the adjoint equations as follows
	\begin{equation}\label{eq:alpha}
		\left\{
		\begin{aligned}
			\drm \alpha(t)&=-\Big[b_x(t)\alpha(t)+\sigma_x(t)\beta(t)
			+\int_{\cE}\eta_x(t,\zeta)\theta(t,\zeta)\nu(\drm\zeta) +\sum_{k=1}^n\gamma^k(t)\lambda^\prime_k(t-)M(t-)
			\\
			&\quad+M(t)f_x(t)\Big]\drm t +\beta(t)\drm W(t)
			+\int_{\cE}\theta(t,\zeta)\tilde{N}(\drm\zeta,\drm t) +\sum_{k=1}^n\xi^k(t)\drm(Y_k(t)-t),\\
			\alpha(T) &=\Phi^\prime(X(T))M(T).
		\end{aligned}
		\right. 
	\end{equation}
	and
	\begin{equation}\label{eq:gamma}
		\left\{
		\begin{aligned}
			\drm\bar{\gamma}(t)&=-\left[\sum_{k=1}^{n}\gamma^k(t)
			\left(\lambda_k(t-)-1\right)+f(t)\right]\drm t +\bar{\beta}(t)\drm W(t)
			+\int_{\cE}\bar{\theta}(t,\zeta)\tilde{N}(\drm\zeta,\drm t)\\
			&\qquad+\sum_{k=1}^{n}\gamma^k(t)\drm(Y_k(t)-t),\\
			\bar{\gamma}(T)&=\Phi(X(T)).
		\end{aligned}
		\right.
	\end{equation}
	Based on Assumption \ref{Ass:1}, by using Lemma~\ref{lem:moment-PQ} and Lemma~\ref{lem:M4}, the terminal conditions $ \alpha(T)=\Phi'(X(T))M(T)$, and $\bar{\gamma}(T)=\Phi(X(T))$ are square-integrable.    Moreover, the adjoint system is triangular. Specifically, equation~\eqref{eq:gamma} can first be solved for $(\bar\gamma,\bar\beta,\bar\theta,\gamma^1,\ldots,\gamma^n)$, after which the resulting processes $\gamma^k$ enter equation~\eqref{eq:alpha}.
	
		We now specify the solution classes of the two adjoint equations. For $p>1$ and $0\leq s<t\leq T$, let $\mathbb S^p(s,t)$ denote the space of $\widehat{\mathbb F}$-adapted c\`adl\`ag processes $\varphi$ such that
		\[
		\|\varphi\|_{\mathbb S^p(s,t)}:=\left\{\mathbb E^{\mathbb Q}\left[\sup_{s\leq r\leq t}|\varphi(r)|^p\right]\right\}^{1/p}<\infty.
		\] 
		Let $\mathbb H^p(s,t)$ denote the space of $\widehat{\mathbb F}$-predictable processes $\psi$ such that
		\[
		\|\psi\|_{\mathbb H^p(s,t)}:=\left\{\mathbb E^{\mathbb Q}\left[\left(\int_s^t|\psi(r)|^2 \mathrm dr\right)^{p/2}\right]\right\}^{1/p}<\infty.
		\] 
		Let $\mathbb J^p(s,t)$ denote the space of predictable integrands $\psi$ such that
		\[
		\|\psi\|_{\mathbb J^p(s,t)}:=\left\{\mathbb E^{\mathbb Q}\left[\left(\int_s^t\int_{\mathcal E}|\psi(r,\zeta)|^2\bar N(\mathrm d\zeta,\mathrm dr)\right)^{p/2}\right]\right\}^{1/p}<\infty.
		\]  
		For $\xi=(\xi^1,\ldots,\xi^n)$, let $\mathbb K^p(s,t)$ denote the space of predictable $\mathbb R^n$-valued processes satisfying
		\[
		\|\xi\|_{\mathbb K^p(s,t)}:=\left\{\mathbb E^{\mathbb Q}\left[\left(\sum_{k=1}^n\int_s^t|\xi^k(r)|^2 \mathrm dY_k(r)\right)^{p/2}\right]\right\}^{1/p}<\infty.
		\]
		We write $\mathbb S^p=\mathbb S^p(0,T)$, and similarly for $\mathbb H^p$, $\mathbb J^p$, and $\mathbb K^p$.

		We now introduce the following lemma, whose proof is presented in Appendix \ref{App:1}, to establish the solution space of the adjoint equations \eqref{eq:alpha} and \eqref{eq:gamma}.
		
		\begin{lemma}\label{lem:socal}
			Under Assumptions~\ref{Ass:1}--\ref{Ass:4}, the adjoint equations
			\eqref{eq:alpha} and \eqref{eq:gamma} admit unique solutions satisfying $
			(\alpha,\beta,\theta,\xi) \in \mathbb S^{3/2}\times\mathbb H^{3/2} \times\mathbb J^{3/2}\times\mathbb K^{3/2}$, and $ 
			(\bar\gamma,\bar\beta,\bar\theta,\gamma)
			\in
			\mathbb S^2\times\mathbb H^2
			\times\mathbb J^2\times\mathbb K^2$, 
			respectively, where
			$\xi=(\xi^1,\ldots,\xi^n)$ and
			$\gamma=(\gamma^1,\ldots,\gamma^n)$.
		\end{lemma}

	The main result of this paper is the following theorem, which establishes a necessary condition for an admissible control to be optimal. 
	\begin{theorem}\label{th:1}
		Suppose Assumptions~\ref{Ass:1}--\ref{Ass:4} hold.
		Let $u(\cdot)$ be an interior optimal admissible control with corresponding state $(X,M)$.
		Let $(\alpha,\beta,\theta,\xi^1,\ldots,\xi^n)$ and $(\bar\gamma,\bar\beta,\bar\theta,\gamma^1,\ldots,\gamma^n)$
		be the unique solutions of \eqref{eq:alpha} and \eqref{eq:gamma}, respectively,
		in the classes presented in Lemma~\ref{lem:socal}.
		Then, for $\mathrm{d}t\otimes\mathrm{d}\mathbb{P}$-a.e.\ $(t,\omega)\in[0,T]\times\Omega$,
		\[
		\mathbb{E}^\mathbb{Q}\Big[\partial_u H\big(t,u(t), X(t), M(t), \alpha(t), \beta(t), \gamma^1(t),\ldots,\gamma^n(t), \theta(t,\cdot)\big)
		\Big|  \cF^{\overrightarrow{Y}}_t\Big] = 0.
		\]
	\end{theorem}
	\begin{proof}
		Applying It\^{o}'s formula to \(\alpha(\cdot)X^l(\cdot)\) and \(\bar{\gamma}(\cdot)M^l(\cdot)\), we have that $$
		\begin{aligned}
			\drm(\alpha(t)  X^l (t) ) = &-X^l   \bigg[ \sum_{k=1}^n \gamma^k  \lambda_k^\prime(t-) M(t-) + Mf_x\bigg] \drm t + \big[ \alpha b_u v + \beta \sigma_u v   +\int_{\cE}\eta_u(t,\zeta)  \theta(t,\zeta)v(t)\\ 
			&\nu(\drm \zeta)\big] \drm t + \{\cdots\}\drm W(t)+\{\cdots\}\tilde{N}(\drm\zeta,\drm t)+ \{\cdots\}\drm(Y_k(t)-t),
		\end{aligned}$$
		and $$
		\begin{aligned}
			\drm (\bar{\gamma}(t)M^l(t))  =& \bigg[- M^l(t-)f  + \sum_{k=1}^n \gamma^k  \lambda^\prime_k(t-) M(t-)X^l  \bigg] \drm t+ \{\cdots\}\drm W(t)\\
			&+\{\cdots\}\tilde{N}(\drm\zeta,\drm t)+ \{\cdots\}\drm(Y_k(t)-t).
		\end{aligned}$$  
			Before taking expectations, we verify the integrability required for
			the localization argument. By Lemmas~\ref{lem:moment-PQ}
			and~\ref{lem:XL}, together with the estimate
			\eqref{eq:socal-ga}, the Cauchy--Schwarz inequality gives $
			\mathbb E^{\mathbb Q}\int_0^T |M(t)\gamma^k(t)X^l(t)| \mathrm dt \leq
			\left( \mathbb E^{\mathbb Q}  	\int_0^T|\gamma^k(t)|^2 \mathrm dt \right)^{1/2}
			\times \left( \mathbb E^{\mathbb Q} \int_0^T|M(t)|^2|X^l(t)|^2 \mathrm dt
			\right)^{1/2} <\infty$. 
			
	Moreover,  applying the Poisson-integrand estimate derived in the proof of Lemma \ref{lem:socal} in Appendix \ref{App:1}, with \(\eta_x\) substituted by \(\eta_u\), we obtain $\left\| \int_0^T \left| \int_{\mathcal E} \eta_u(t,\zeta)\theta(t,\zeta)
			\nu(\mathrm d\zeta) \right| \mathrm dt \right\|_{L^{3/2}(\mathbb Q)}
			\leq 	C\left(T^{1/2}+T^{1/3}\right) \|\theta\|_{\mathbb J^{3/2}(0,T)}
			<\infty. $
			Since $v$ is bounded, these and the remaining finite-variation terms
			are integrable under Assumption~\ref{Ass:1}.
			
			By the definition of local martingale, we can define a sequence of stopping times $(\sigma_n)_{n\geq1}$ such that 	$\sigma_n\uparrow\infty$ and every local martingale term appearing in the two product formulas, when stopped at $\sigma_n$, is a true martingale. Let $\tau_n:=\sigma_n\wedge T$.  Then $\tau_n\uparrow T$, and all the corresponding stopped local martingale terms have zero expectation.  Applying the BDG and Kunita inequalities to \eqref{eq:Psi} stopped at $\tau_n$, and using the uniform bounds in Assumption~\ref{Ass:1}, we obtain $
			\mathbb E^{\mathbb Q}|X^l(\tau_n)|^4\leq C\int_0^T\mathbb E^{\mathbb Q}\left[|X^l(s)|^4+|v(s)|^4\right] \mathrm ds$,  where $C$ is independent of $n$. The right-hand side is finite based on Lemma~\ref{lem:XL} and the boundedness of $v$. Hence, $\sup_{n\geq1}\mathbb E^{\mathbb Q}|X^l(\tau_n)|^4<\infty$. 
			Together with $\alpha\in\mathbb S^{3/2}$, applying the H\"older's inequality yields
			$ \sup_{n\geq1}\mathbb E^{\mathbb Q}|\alpha(\tau_n)X^l(\tau_n)|^{12/11}<\infty$. 
			Therefore, $\{\alpha(\tau_n)X^l(\tau_n)\}_{n\geq1}$ is uniformly integrable.

			Furthermore, based on \eqref{eq:varphi}, combining Assumption~\ref{Ass:4} and Lemma~\ref{lem:XL},  we obtain that the integrand in the equation for
			$M^l$ is square-integrable. Hence, $M^l$ is a square-integrable
			martingale. Since $\bar\gamma\in\mathbb S^2$, the families $\{|M^l(\tau_n)|^2\}_{n\geq1}$ and $\{|\bar\gamma(\tau_n)|^2\}_{n\geq1}$ are uniformly integrable. Since  $|\bar\gamma(\tau_n)M^l(\tau_n)| \leq
			\frac12|\bar\gamma(\tau_n)|^2 + \frac12|M^l(\tau_n)|^2$, then 
			$\{\bar\gamma(\tau_n)M^l(\tau_n)\}_{n\geq1}$ is uniformly integrable.
			
			We then apply the stopped product identities and let
			$\tau_n\uparrow T$. The stopped local martingales have zero
			expectation, while the preceding integrability and
			uniform-integrability estimates justify passage to the limit.

		Taking expectations of the above equations, we obtain
		\begin{equation}\label{eq:E1}
			\begin{aligned}
				&\EE^\QQ[\Phi^\prime(X(T))M(T)X^l(T)]  = \EE^\QQ \biggl\{ \int_0^T -X^l   \biggl[ \sum_{k=1}^n \gamma^k  \lambda_k^\prime(t-)M(t-) + M  f_x   \biggr] \drm t \\
				&+\int_0^T  \bigl[ \alpha  b_u  v   + \beta  \sigma_u  v  +\int_\cE \eta_u(t,\zeta)  \theta(t,\zeta)v(t)\nu(\drm \zeta) \bigr] \drm t \biggr\},
			\end{aligned}
		\end{equation}
		and
		\begin{equation}\label{eq:E2}
			\begin{aligned}
				&\EE^\QQ\left[\Phi(X(T))M^l(T)\right] =\EE^\QQ \biggl[ \int_0^T \Bigl(
				-M^l(t-)f+ \sum_{k=1}^n \gamma^k \lambda^\prime_k(t-)M(t-)X^l
				\Bigr)\drm t \biggr].
			\end{aligned}
		\end{equation}
		
		Substituting \eqref{eq:E1} and \eqref{eq:E2} into \eqref{eq:FD}, we have that
		\begin{small}
			\begin{align*}
				&0= \frac{\drm }{\drm \ep} J(\ep)\left|_{\ep=0}\right.\\
				&=\EE^\QQ \biggl[ \int_0^T \Bigl(
				-M^l(t-)f+ \sum_{k=1}^n \gamma^k \lambda^\prime_k(t-)M(t-)X^l
				\Bigr)\drm t \biggr]+\EE^\QQ\left[\int_0^T f  M^l(t-)+M  \left(f_x  X^l  +f_u  v  \right)\drm t\right]\\
				& +\EE^\QQ\bigg\{\int_0^T(-X^l  )\left[\sum_{k=1}^n\gamma^k   \lambda_k^\prime(t-)M(t-)+M  f_x  \right]  + \alpha  b_u  v  +\beta  \sigma_u  v  +\int_\cE \eta_u(t,\zeta)\theta(t,\zeta)v  \nu(\drm \zeta) \drm t\bigg\} .
			\end{align*}
		\end{small}
		After cancellation, we obtain
		\[
		0=\mathbb E^{\mathbb Q}\left\{\int_0^T\partial_uH\bigl(t,u(t),X(t),M(t),\alpha(t),\beta(t),\gamma^1(t),\ldots,\gamma^n(t),\theta(t,\cdot)\bigr)v(t) \mathrm dt\right\}.
		\]
		Since $v$ is an arbitrary bounded $\mathbb F^{\overrightarrow Y}$‑predictable  variation, 
		\[
		\mathbb E^{\mathbb Q}\left[\left.\partial_uH\bigl(t,u(t),X(t),M(t),\alpha(t),\beta(t),\gamma^1(t),\ldots,\gamma^n(t),\theta(t,\cdot)\bigr)\right|\mathcal F_{t-}^{\overrightarrow Y}\right]=0,\quad \mathrm dt\otimes\mathrm d\mathbb Q\text{-a.e.}
		\]
		Note that $\mathcal F_{t-}^{\overrightarrow Y}$ may be
		replaced by $\mathcal F_t^{\overrightarrow Y}$ in the preceding
		conditional expectation for $\mathrm dt\otimes\mathrm d\mathbb Q$-a.e. $(t,\omega)$. Therefore, the equivalence $\mathbb P\sim\mathbb Q$ implies that the result holds $\mathrm dt\otimes\mathrm d\mathbb P$-a.e.

	\end{proof}
	
	\begin{remark}
		Theorem~\ref{th:1} provides a first-order necessary condition 
		for local optimality under partial observations. Throughout, we work with a convex control domain and employ additive perturbations around an interior optimal control. 
	\end{remark}

\section{A Linear-Quadratic Illustration}\label{sec:app}
We illustrate the necessary condition in Theorem~\ref{th:1} using a partially observed linear--quadratic control model. We first derive the conditional stationarity condition and examine the compatibility of a linear ansatz for the first adjoint process with the coupled adjoint system. We then approximate the resulting partial-information adjoint equations and the conditional stationarity relation numerically by a simulation-based fixed-point scheme. The numerical experiment further examines how the observation-induced adjoint coupling affects the resulting control under different observation sensitivities. Throughout this section, we take $\mathbb U=\mathbb R$.

\subsection{Problem Formulation}
We consider a stylized control problem motivated by trading, rather than a complete optimal-execution model. By omitting separate inventory and cash accounts and a transaction-price process, we focus on the partially observed LQ tracking problem.

Let $X(t)$ denote a latent price factor satisfying
\begin{equation}\label{eq:appl-state}
	\drm X(t)=\big[\vartheta(\bar X-X(t))-\rho u(t)\big]\drm t+\sigma\drm W(t)+\int_\cE\eta(\zeta)\tilde N(\drm\zeta,\drm t),
\end{equation}
with $X(0)=x\in\RR$. Here, $\vartheta>0$ is the mean-reversion rate and $\bar X$ is the long-run price level. The control $u(t)$ is the signed trading rate with $u(t)>0$ denoting a sale, and $\rho>0$ measures its linear price impact. The Brownian term represents continuous price fluctuations, whereas the compensated jump integral captures exogenous market shocks such as macroeconomic news arrivals or sudden liquidity events.

Provided that $\eta\in L^2(\cE,\nu)\cap L^4(\cE,\nu)$, the coefficients $b(t,x,u)=\vartheta(\bar X-x)-\rho u$, $\sigma(t,x,u)\equiv\sigma$, and $\eta(t,x,u,\zeta)=\eta(\zeta)$ satisfy Assumption~\ref{Ass:1}. Moreover, the controller observes only the counting process $\overrightarrow{Y}(t)=(Y_1(t),\ldots,Y_n(t))^\top$ from \eqref{eq:observe}, whose $k$-th component records market events at price level $y_k$ with intensity $\lambda_k(X(t-),t-)$ satisfying Assumptions~\ref{Ass:3}--\ref{Ass:4}.

Assumption~\ref{Ass:2} imposes mutual independence among the primitive noises $N_1,\ldots,N_n$, $W$, and $\bar N$. This does not preclude feedback through the control, since $u$ is $\mathbb F^{\overrightarrow Y}$-predictable and the controlled state may depend on the observation history.

The objective is to minimize the expected LQ cost
\begin{equation}\label{eq:appl-cost}
	J(u)=\EE^\PP\left[X(T)^2-k_1X(T)+k_2\int_0^T u(t)^2\drm t\right].
\end{equation}
Here, the terminal penalty $\Phi(x)=x^2-k_1x=(x-\frac{k_1}{2})^2-\frac{k_1^2}{4}$, so $k_1/2$ is the terminal target level. The running cost is $f(t,x,u)=k_2u^2$, where $k_2>0$ penalizes excessively high trading rates and serves as a reduced-form proxy for liquidity and temporary-impact costs.

To handle partial observability, we employ the measure change introduced in Subsection~\ref{sub:co}. The augmented state dynamics are given by \eqref{eq:Qstate}, where $X(t)$ evolves according to \eqref{eq:appl-state}. The corresponding cost functional under $\QQ$ is
\[
J(u(\cdot))=\EE^\QQ\left[\left(X(T)^2-k_1X(T)\right)M(T)+k_2\int_0^T u^2M\drm t\right].
\]

Consider an  optimal control $u$ to which Theorem~\ref{th:1} applies, together with its associated state and adjoint processes. Then the Hamiltonian under $\QQ$ is $
H(t,u,x,m,\alpha,\beta,\gamma^1,\ldots,\gamma^n,\theta)
=\left[\vartheta(\bar X-x)-\rho u\right]\alpha+\sigma\beta+\sum_{k=1}^n\gamma^k\left(\lambda_k(x,t)-1\right)m
+\int_{\mathcal E}\eta(\zeta)\theta(t,\zeta)\nu(\drm\zeta)+mk_2u^2.$ 
The corresponding adjoint processes are governed by
\begin{equation}\label{eq:FAalpha}
	\left\{
	\begin{aligned}
		\drm\alpha(t)&=-\Big[-\vartheta\alpha+\sum_{k=1}^n\gamma^k\lambda_k^\prime(t-)M(t-)\Big]\drm t+\beta\drm W(t)+\int_{\mathcal E}\theta(t,\zeta)\tilde N(\drm\zeta,\drm t)\\
		&\qquad+\sum_{k=1}^n\xi^k\drm(Y_k(t)-t),\\
		\alpha(T)&=2X(T)M(T)-k_1M(T),
	\end{aligned}
	\right.
\end{equation}
and
\begin{equation}\label{eq:barga}
	\left\{
	\begin{aligned}
		\drm\bar\gamma(t)&=-\left[\sum_{k=1}^{n}\gamma^k\left(\lambda_k(t-)-1\right)+k_2u^2\right]\drm t+\sum_{k=1}^{n}\gamma^k\drm(Y_k(t)-t)+\bar\beta(t)\drm W(t)\\
		&\qquad+\int_\cE\bar\theta(t,\zeta)\tilde N(\drm\zeta,\drm t),\\
		\bar\gamma(T)&=X^2(T)-k_1X(T).
	\end{aligned}
	\right.
\end{equation}

\subsection{Conditional Stationarity and Compatibility of the Linear Ansatz}
Since the admissible control is $\mathbb{F}^{\overrightarrow{Y}}$-predictable, the stationarity condition holds in the predictable sense:
\[
\mathbb{E}^{\mathbb{Q}} \left[-\rho\alpha(t-)+2k_2u(t)M(t-)\mid \mathcal{F}_{t-}^{\overrightarrow{Y}}\right]=0,\qquad dt\otimes d\mathbb{Q}\text{-a.e.}
\]
As $u(t)$ is $\mathcal{F}_{t-}^{\overrightarrow{Y}}$-measurable, we rearrange to obtain
\[
-\rho \mathbb{E}^{\mathbb{Q}} \left[\alpha(t-)\mid\mathcal{F}_{t-}^{\overrightarrow{Y}}\right]+2k_2u(t) \mathbb{E}^{\mathbb{Q}} \left[M(t-)\mid\mathcal{F}_{t-}^{\overrightarrow{Y}}\right]=0.
\]
Since $M(t-)>0$ almost surely, its corresponding conditional expectation is strictly positive almost surely. We therefore arrive at
\begin{equation}\label{eq:optimal-u}
	u(t)=\frac{\rho}{2k_2}\frac{\mathbb{E}^{\mathbb{Q}} \left[\alpha(t-)\mid\mathcal{F}_{t-}^{\overrightarrow{Y}}\right]}{\mathbb{E}^{\mathbb{Q}} \left[M(t-)\mid\mathcal{F}_{t-}^{\overrightarrow{Y}}\right]},\qquad dt\otimes d\mathbb{Q}\text{-a.e.}
\end{equation}

To examine whether deterministic linear decoupling is compatible with the partial-information adjoint system, suppose that
\begin{equation}\label{eq:alpha1}
	\alpha(t)=M(t)\bigl(P(t)X(t)+Q(t)\bigr),
\end{equation}
where $P,Q\in C^1([0,T];\mathbb R)$ are deterministic and satisfy $P(T)=2$, and $Q(T)=-k_1$.
Applying It\^o's product formula to \eqref{eq:alpha1} yields
\begin{equation}\label{eq:dalpha-ansatz}
	\begin{aligned}
		\drm\alpha(t)&=M(t-)\Big[\bigl(P'(t)-\vartheta P(t)\bigr)X(t-)+Q'(t)+\vartheta\bar X P(t)-\rho P(t)u(t)\Big]\drm t\\
		&\quad+\sigma M(t-)P(t)\drm W(t) +\int_{\cE}M(t-)P(t)\eta(\zeta) \widetilde N(\drm\zeta,\drm t)\\
		&\quad+M(t-)\bigl[P(t)X(t-)+Q(t)\bigr]\sum_{k=1}^{n}\bigl(\lambda_k(X(t-),t-)-1\bigr)\drm\bigl(Y_k(t)-t\bigr).
	\end{aligned}
\end{equation}
Under $\mathbb Q$, the processes $Y_1,\ldots,Y_n$ are independent unit-rate Poisson processes, whose jumps are almost surely disjoint from the jumps of the state-jump random measure. Hence, whenever $\Delta Y_k(t)=1$, the state $X$ does not jump simultaneously, and the jump in $\alpha$ is induced exclusively by the jump of $M$. Matching the observation-martingale terms in \eqref{eq:dalpha-ansatz} and \eqref{eq:FAalpha} gives for $k=1,\cdots,n$,
\begin{equation}\label{eq:xi-match}
	\xi^k(t)=M(t-)\bigl[P(t)X(t-)+Q(t)\bigr]\bigl(\lambda_k(X(t-),t-)-1\bigr).
\end{equation}
Matching the Brownian and state-jump martingale terms yields
\begin{align}\label{eq:beeta}
	\beta(t) =\sigma M(t-)P(t), \quad
	\theta(t,\zeta)=M(t-)P(t)\eta(\zeta).
\end{align}
It remains to match the drift components. Comparing \eqref{eq:dalpha-ansatz} with \eqref{eq:FAalpha} produces
\[
\begin{aligned}
	0=&M(t-)\Big[\bigl(P'(t)-\vartheta P(t)\bigr)X(t-)+Q'(t)+\vartheta\bar X P(t)-\rho P(t)u(t)\Big]\\
	&+M(t-)\sum_{k=1}^{n}\gamma^k(t)\lambda_k^\prime(X(t-),t-)-\vartheta\alpha(t-),
\end{aligned}
\]
which holds $dt\otimes d\mathbb Q$-a.e. Substituting \eqref{eq:alpha1}, we obtain
\begin{equation}\label{eq:matching}
	\begin{aligned}
		0=&\bigl(P'(t)-2\vartheta P(t)\bigr)X(t-)+Q'(t)-\vartheta Q(t)+\vartheta\bar X P(t)-\rho P(t)u(t)\\
		&+\sum_{k=1}^{n}\gamma^k(t)\lambda_k^\prime(X(t-),t-),
	\end{aligned}
\end{equation}
again $dt\otimes d\mathbb Q$-a.e..

Under the ansatz~\eqref{eq:alpha1}, combining the conditional stationarity condition~\eqref{eq:optimal-u} with the conditional Bayes formula yields
\begin{equation}\label{eq:ansatz-induced-control}
	\begin{aligned}
		u(t)&=\frac{\rho}{2k_2}\frac{\mathbb E^{\mathbb Q} \left[M(t-)\bigl(P(t)X(t-)+Q(t)\bigr)\mid\mathcal F_{t-}^{\overrightarrow Y}\right]}{\mathbb E^{\mathbb Q} \left[M(t-)\mid\mathcal F_{t-}^{\overrightarrow Y}\right]}\\
		&=\frac{\rho}{2k_2}\left(P(t)\widehat X(t-)+Q(t)\right),
	\end{aligned}
\end{equation}
where $\widehat X(t-):=\mathbb E^{\mathbb P} \left[X(t-)\mid\mathcal F_{t-}^{\overrightarrow Y}\right].$

Thus, even under this linear adjoint ansatz, the resulting control depends on the filtered state rather than directly on the unobserved state $X(t-)$.
Substituting \eqref{eq:ansatz-induced-control} into \eqref{eq:matching} gives
\begin{equation}\label{eq:matching-filtered}
	\begin{aligned}
		0=&\bigl(P'(t)-2\vartheta P(t)\bigr)X(t-)+Q'(t)-\vartheta Q(t)+\vartheta\bar X P(t)\\
		&-\frac{\rho^2}{2k_2}P(t)\bigl(P(t)\widehat X(t-)+Q(t)\bigr)+\sum_{k=1}^{n}\gamma^k(t)\lambda^\prime_k(X(t-),t-),
	\end{aligned}
\end{equation}
holding $dt\otimes d\mathbb Q$-a.e..
Equation~\eqref{eq:matching-filtered} constitutes a necessary compatibility condition imposed by the first adjoint equation on the deterministic linear ansatz~\eqref{eq:alpha1}. It involves both the filtered state $\widehat X$ and the observation-martingale coefficients $\gamma^k$ associated with the adjoint equation for the likelihood-ratio process $M$. Neither quantity can be determined from a deterministic system for $P$ and $Q$ alone. Therefore, unless additional structural identities hold, the partial-information adjoint system does not reduce to closed deterministic equations for $P$ and $Q$. For this reason, we directly approximate the coupled adjoint system and conditional stationarity relation in the numerical implementation presented below.

\subsection{Numerical Implementation and Results}
Let $t_i=i\Delta$ with $i=0,\ldots,N$, and define $
a_i=\frac{\alpha(t_i)}{M(t_i)}$, and $G_i=\bar\gamma(t_i)$. 
 The terminal values are $a_N=2X_N-k_1$ and $G_N=X_N^2-k_1X_N$. Under the unit-rate reference measure, we use the backward recursions
\begin{equation}\label{eq:num-G}
	G_i\approx\widehat{\EE}_i^{\QQ}[G_{i+1}]
	+\Delta\left[\sum_{k=1}^{n}\gamma_i^k\bigl(\lambda_k(X_i)-1\bigr)+k_2u_i^2\right],
\end{equation}
and
\begin{equation}\label{eq:num-a}
	a_i\approx
	\frac{\widehat{\EE}_i^{\QQ}[L_{i+1}a_{i+1}]
		+\Delta\sum_{k=1}^{n}\gamma_i^k\lambda_k'(X_i)}
	{1+\vartheta\Delta},
	\qquad
	L_{i+1}=\frac{M(t_{i+1})}{M(t_i)}.
\end{equation}
Thus, the coupling term in \eqref{eq:matching} is retained at each backward step. Conditional expectations are estimated via regression using simulated state and filter moment samples.  

For a given observation-adapted policy $u$, define
\begin{equation}\label{eq:stationarity-map}
	\mathcal T(u)_i=
	\frac{\rho}{2k_2}
	\frac{\EE^{\QQ}[M_i a_i\mid\cF_{t_i}^{\overrightarrow Y}]}
	{\EE^{\QQ}[M_i\mid\cF_{t_i}^{\overrightarrow Y}]}
	=
	\frac{\rho}{2k_2}\EE^{\PP}[a_i\mid\cF_{t_i}^{\overrightarrow Y}].
\end{equation}
Starting from the initialization $u^{(0)}$, we adopt the following damped update
\begin{equation}\label{eq:fixed-point-update}
	u^{(m+1)}=(1-\omega_m)u^{(m)}+\omega_m\widehat{\mathcal T}(u^{(m)}),
\end{equation}
with initial damping $\omega_0=0.5$.  The iteration is used only to approximate the stationarity condition; the objective value is evaluated separately.

The fixed-point residual is evaluated on separate validation paths via
\begin{equation}\label{eq:fp-residual}
	r_{\rm stat}(u)=\frac{\|u-\widehat{\mathcal T}(u)\|_2}{\|u\|_2+\|\widehat{\mathcal T}(u)\|_2},
\end{equation}
where the norms denote empirical root-mean-square norms across paths and decision times. The damping, acceptance, and stopping rules are given in Appendix~\ref{App:2}.

\paragraph{Simulation procedure.} 

 The initial policy \(u^{(0)}\) is the certainty-equivalent (CE) policy obtained from the full-information LQ feedback associated with \eqref{eq:appl-state}--\eqref{eq:appl-cost}, with the unobserved state replaced by its particle-filter mean. It is used only for initialization and benchmarking.  Given a current policy, we simulate the controlled state and marked observations under $\mathbb{P}$. At each time step, the control is evaluated prior to processing the new observation increment, following which the particle filter is updated using the observed counts.

 The simulated states serve as design points for one-step branches under $\mathbb{Q}$. These branches supply the targets for the backward recursions \eqref{eq:num-G} and \eqref{eq:num-a}. After projection onto the observation-adapted basis, the resulting estimate $\widehat{\mathcal T}(u)$ is inserted into \eqref{eq:fixed-point-update}. The residual is evaluated on separate validation paths.
 
 The final policy is evaluated on fresh physical-measure paths. Its cost is estimated from the terminal penalty and discrete running cost. Common physical noise is used for paired comparison against the CE policy. The particle filter employs an independent internal Monte Carlo stream; conditional on this stream, the computed control constitutes a non-anticipative function of the observed history. Further details on the filter and backward projection are provided in Appendix~\ref{App:2}.

\paragraph{Simulation configuration.}

The observation intensities are
\begin{equation*} 
	\lambda_k(x)=\bar a\left[(1-\varepsilon_\lambda)
	\frac{\exp\{-\beta_p(y_k-x)^2\}}
	{\sum_{j=1}^{n}\exp\{-\beta_p(y_j-x)^2\}}
	+\frac{\varepsilon_\lambda}{n}\right],\qquad k=1,\ldots,n,
\end{equation*}
with $\varepsilon_\lambda=10^{-6}$. Since $\sum_{k=1}^{n}\lambda_k(x)=\bar a$, the total arrival rate is fixed, whereas $\beta_p$ governs the strength with which the distribution of observation marks depends on the latent state. When $\beta_p=0$, the marks become state-independent and $\lambda_k'(x)=0$ for all $k$.

A concrete simulation configuration is $T=2$, $\Delta=0.02$, $x_0=1.4$, $\bar X=1$, $\vartheta=0.5$, $\sigma=0.15$, $\rho=0.4$, $k_1=2$, and $k_2=0.5$. 
There are $n=9$ observation levels $y_k=0.2,0.4,\ldots,1.8$, with $\bar a=50$ and nominal $\beta_p=15$. The compound-Poisson state jump has intensity $\lambda_J=20$, upward-jump probability $p_{\rm up}=0.65$, and single-jump variance $V_J=0.0225$. Its two jump sizes are chosen so that the jump mean is zero.  The particle filter employs $128$ particles. Each main backward calculation uses $3000$ design paths with $40$ conditional branches. At each fixed-point update, the projected targets from three backward passes with independent branch simulations are averaged.  Validation is performed on $4000$ paths, while the cost comparison relies on $12000$ fresh physical-measure paths.

\begin{figure}[htbp]
	\centering
	\includegraphics[width=0.78\textwidth]{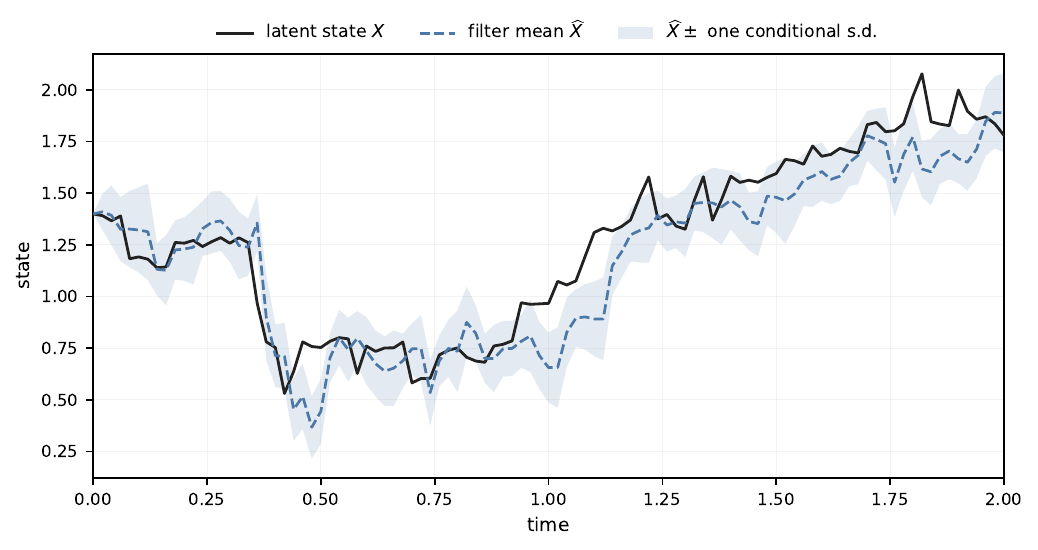}
	\caption{One simulated state trajectory and its particle-filter estimate
		under the SMP fixed-point iterate at the nominal calibration $\beta_p=15$.
		The solid line denotes the latent state $X$, the dashed line the
		particle-filter conditional mean $\widehat X$, and the shaded region one
		conditional standard deviation around the filter mean.}
	\label{fig:nominal-path}
\end{figure}

\paragraph{Numerical results.}
At the nominal value $\beta_p=15$, one damped update reduces the validation stationarity residual from $0.699\%$ to $0.430\%$. Three independent backward passes yield holdout residuals of $0.440\%$, $0.682\%$, and $1.526\%$, with a mean of $0.883\%$. The corresponding pooled one-step residuals are $1.468\%$ for the first adjoint recursion and $0.670\%$ for the second.

Table~\ref{tab:nominal-benchmarks} places the SMP--CE comparison on the scale of
the control problem. The full-information Riccati policy is included only as an
informational benchmark and is not admissible under partial observation.
\begin{table}[htbp]
	\centering
	\small
	\setlength{\tabcolsep}{5pt}
	\caption{Controller comparison at $\beta_p=15$ using $12000$ fresh
		physical-measure paths. Differences from CE use common physical noise.}
	\label{tab:nominal-benchmarks}
	\begin{tabular}{lrrc}
		\toprule
		Policy & $J(u)$ & $J(u)-J(u^{\rm CE})$ & Paired 95\% CI \\
		\midrule
		SMP fixed-point iterate
		& $-0.609074$ & $-0.000003$ & $[-0.000025,\ 0.000019]$ \\
		CE--Riccati
		& $-0.609072$ & $0$ & -- \\
		Full-information Riccati
		& $-0.613392$ & $-0.004320$ & $[-0.004481,\ -0.004159]$ \\
		Zero control
		& $-0.559577$ & $0.049495$ & $[0.047765,\ 0.051225]$ \\
		\bottomrule
	\end{tabular}
\end{table}
At the nominal calibration, both observation-adapted policies yield lower simulated costs than zero control, while the full-information Riccati benchmark
attains a still lower simulated cost. The SMP and CE costs are not
distinguishable at the Monte Carlo precision used here.

We next vary the observation-sensitivity parameter over $\beta_p\in\{0,5,10,15,20\}$,  while keeping the total observation intensity fixed at $\bar a=50$. A separate fixed-point calculation is performed for each value of $\beta_p$, starting from the corresponding CE policy. To measure the difference between the resulting policy map and its CE
initialization, we define
\begin{equation}\label{eq:policy-gap}
	G_u(\beta_p)=\left[\frac{\EE\int_0^T|u_t^{\rm SMP}-u_t^{\rm CE}|^2 \drm t}
	{\EE\int_0^T|u_t^{\rm CE}|^2 \drm t}\right]^{1/2}.
\end{equation}
Both policy maps in \eqref{eq:policy-gap} are evaluated along the same independent CE-generated filter trajectories. We also compute $\Delta J(\beta_p)=J(u^{\rm SMP})-J(u^{\rm CE})$ using common physical noise.

\begin{table}[htbp]
	\centering
	\small
	\setlength{\tabcolsep}{4pt}
	\caption{Sensitivity to the observation parameter $\beta_p$. Filter RMSE is the mean across evaluation paths of the pathwise time-RMSE under the CE policy. The cost difference and its confidence interval are multiplied by $10^5$.}
	\label{tab:beta-sensitivity}
	\begin{tabular}{cccccc}
		\toprule
		$\beta_p$ & Filter RMSE & Coupling RMS & $G_u$ (\%)
		& $10^5\Delta J$ & $10^5$ paired 95\% CI \\
		\midrule
		$0$  & $0.4769$ & $0.0000$ & $0.245$ & $-0.215$ & $[-0.326,-0.104]$ \\
		$5$  & $0.2067$ & $0.0270$ & $0.514$ & $-0.211$ & $[-2.781, 2.360]$ \\
		$10$ & $0.1823$ & $0.0393$ & $0.657$ & $0.542$  & $[-1.593, 2.677]$ \\
		$15$ & $0.1718$ & $0.0475$ & $0.698$ & $-0.282$ & $[-2.504, 1.940]$ \\
		$20$ & $0.1659$ & $0.0542$ & $0.713$ & $0.032$  & $[-2.048, 2.112]$ \\
		\bottomrule
	\end{tabular}
\end{table}

At $\beta_p=0$, the observation marks are state-independent and
$\lambda_k'(x)=0$ for every $k$, so the adjoint coupling vanishes.
The small difference between the implemented SMP and CE policies
in this case is therefore not due to this coupling. Across the tested grid,
the CE filter RMSE decreases as $\beta_p$ increases, while the estimated
coupling is nonzero for every positive $\beta_p$. The relative
$L^2$ policy-map gap remains below $1\%$ throughout the grid. For
$\beta_p=5,10,15,$ and $20$, the paired confidence intervals for
$\Delta J$ contain zero. At $\beta_p=0$, the paired interval does not contain zero, although the  absolute cost difference remains very small. Thus, under the reported calibration, the SMP iterate remains close to the CE policy even when the observation-induced coupling is nonzero.

Numerical-resolution sensitivity at the nominal value is reported in
Appendix~\ref{App:3}.

\section{Conclusion}\label{sec:conclusion}

We derive a stochastic maximum principle for partially observed
jump-diffusion systems with multivariate counting-process observations
and state-dependent intensities. Using a reference probability measure
and augmenting the state with the counting-process likelihood ratio, we
obtain a joint state--likelihood variational system and a conditional
Hamiltonian stationarity condition. The state dependence of the
observation intensities gives rise to an additional coupling term in
the adjoint system.

For the linear-quadratic illustration, we derive the compatibility condition required by a deterministic linear adjoint ansatz and numerically approximate the resulting coupled partial-information adjoint system. The numerical experiments illustrate the effect of state-dependent observation marks on the adjoint coupling.


\bibliographystyle{plainnat}
\bibliography{sn-bibliography}

\begin{appendices}
	\section{The Proof of Theory given in Section \ref{sec:SMP}}\label{App:1}
	\noindent\textbf{The proof of Lemma \ref{lem:moment-PQ}:} 
	The finite-variation term of $\drm M$ is
	$-M(t-)\sum_k(\lambda_k-1)\drm t$, whereas its purely jump term is
	$M(t-)\sum_k(\lambda_k-1)\drm Y_k(t)$.
	Note that $M$ has only jumps coming from the $\{Y_k\}$. Specifically, when $Y_k$ 
	jumps, $M$ jumps by $\Delta M(t) = M(t-)(\lambda_k(X(t-),t-) - 1)$. Since \(M(0)=1\), \(M\) remains positive between jumps, and at a jump time of \(Y_k\), $
	M(t)=M(t-)\lambda_k(X(t-),t-)>0$  because \(\lambda_k\ge\delta_\lambda>0\). Hence \(M(t)>0\) for all \(t\in[0,T]\), and \(M^q\) is well defined for every \(q\in\mathbb R\). Applying It\^{o}'s formula for jump processes, $
	\drm M^q(t) = q M^{q-1}(t-)\Big[-M(t-)\sum_k(\lambda_k-1)\Big]\drm t + M^q(t-)\sum_k(\lambda_k^q - 1) \drm Y_k(t).$
	and rewriting $\drm Y_k = \drm(Y_k - t) + \drm t$, we have
	\begin{equation}\label{eq:Mkappa-SDE}
		\drm M^q(t) = M^q(t-)\sum_{k=1}^n(\lambda_k^q - 1) \drm(Y_k(t)-t) + M^q(t-) \Psi_q(t) \drm t,
	\end{equation}
	where $\Psi_q(t) := \sum_{k=1}^n[\lambda_k^q - 1 - q(\lambda_k - 1)]$. Define the local martingale $I_q(t):=\sum_{k=1}^n\int_0^t
	M^q(s-)\bigl(\lambda_k^q(s)-1\bigr) \mathrm d\bigl(Y_k(s)-s\bigr)$.
	Then \eqref{eq:Mkappa-SDE} can be written as $
	M^q(t)=1+I_q(t)+\int_0^t M^q(s-)\Psi_q(s) \mathrm ds$.
	
	By Assumption \ref{Ass:3} and \ref{Ass:4}, the intensity satisfies \(\lambda_k(X(t-),t-)\in[\delta_\lambda, K_2]\). Since the functions \(\lambda^q - 1 - q(\lambda - 1)\) and \(\lambda^q - 1\) are continuous on the compact interval \([\delta_\lambda, K_2]\) regardless of the sign of \(q\), \(\Psi_q(t)\) and \(|\lambda^q - 1|\) are uniformly bounded by constants \(C_q\) and \(C\), which depend only on \(q\), \(\delta_\lambda\), \(K_2\), and \(n\). We denote both bounds by \(C_q\).

	For each $N \ge 1$, we  define the stopping time as follows $ 
	\tau_N := \inf\left\{t\ge 0: M(t) > N \text{ or } M(t) < \frac{1}{N}\right\}\wedge T$. 
	By the definition of \(\tau_N\), $N^{-1}\le M(s-)\le N$ for any $ s\le\tau_N$,  and hence \(M^{2q}(s-)\le N^{2|q|}\). Thus, \(\mathbb E^{\mathbb Q}\left[\left\langle I_q^{\tau_N}\right\rangle_T\right]\le nC_q^2N^{2|q|}T<\infty.\) Therefore, \(I_q^{\tau_N}\) is a square‑integrable \(\mathbb Q\)-martingale, and $\EE^\QQ[I_q^{\tau_N}(t)]=0$.  Taking $\QQ$-expectation in 
	\eqref{eq:Mkappa-SDE} stopped at $\tau_N$ gives
	$\EE^\QQ\bigl[M^q(t\wedge\tau_N)\bigr] = 1 +
	\EE^\QQ\int_0^{t\wedge\tau_N} M^q(s-)\Psi_q(s)\drm s.$
	Since $M^q(s-)>0$ and $|\Psi_q(s)|\le C_q$ uniformly, we obtain
	$\EE^\QQ\bigl[M^q(t\wedge\tau_N)\bigr] 
	\le 1 + C_q\int_0^t \EE^\QQ\bigl[M^q(s\wedge\tau_N)\bigr] \drm s.
	$ By Gronwall's inequality, 
	$\sup_{0\le t\le T}\EE^\QQ[M^q(t\wedge\tau_N)] \le e^{C_q T}$.

	Under \(\mathbb Q\), each \(Y_k\) is a unit-rate Poisson process
	and therefore has finitely many jumps on \([0,T]\), almost surely.
	Moreover, since \(\lambda_k\in[\delta_\lambda,K_2]\), the product representation
	of \(M\) implies that both \(M\) and \(M^{-1}\) are pathwise bounded
	on \([0,T]\).  Thus,  $\tau_N\uparrow T$, $\QQ$-a.s. Then, by using the Fatou's lemma, we have $\sup_{0\le t\le T}\EE^\QQ [M^q(t)] \le \liminf_{N\to\infty}\sup_{0\le t\le T}\EE^\QQ [M^q(t\wedge\tau_N)] \le e^{C_q T} < \infty.$
	This establishes \eqref{eq:Mq}.

	Applying \eqref{eq:Mq} with $q$ replaced by $2q$ gives $\EE^\QQ\int_0^T M^{2q}(s) \drm s \le T e^{C_{2q}T} < \infty$. Since the compensator of $Y_k$ under measure $\mathbb{Q}$ is given by $t$, the predictable quadratic variation satisfies $\langle I_q\rangle(t) = \int_0^t M^{2q}(s-)\sum_{k=1}^n(\lambda_k^q-1)^2 \drm s \le nC_q^2\int_0^T M^{2q}(s) \drm s$. Hence, $\mathbb E^{\mathbb Q}\langle I_q\rangle_T<\infty$. Thus, the local martingale \(I_q\) is a square-integrable \(\mathbb Q\)-martingale. In particular, taking \(q=1\) gives \(\Psi_1=0\), and hence \(M\) is a true \(\mathbb Q\)-martingale with \(\mathbb E^{\mathbb Q}[M(t)]=1\).
	
	Choose $r=p_0$ given in the definition of $\cU_{ad}$, by using H\"older's inequality, we have $\EE^\QQ \int_0^T |u|^4\drm t = \EE^\PP\big[M^{-1}(T) \int_0^T|u|^4\drm t\big] \le (\EE^\PP[M^{-\frac{r}{r-1}}(T)])^{\frac{r-1}{r}} (\EE^\PP[(\int_0^T |u|^4\drm t)^r])^{\frac{1}{r}} \le T^{(r-1)/r} (\EE^\PP[M^{-\frac{r}{r-1}}(T)])^{\frac{r-1}{r}} (\EE^\PP[\int_0^T |u|^{4r}\drm t])^{\frac{1}{r}}$.
	
	Since $\EE^\PP[M^q(t)]=\EE^\QQ[M^{q+1}(t)]$, the moment estimate \eqref{eq:Mq} implies $\sup_t \EE^\PP [M^q(t)] < \infty$ for all $q\in\RR$. In particular, $\EE^\PP [M(T)^{-r/(r-1)}] < \infty$. Moreover, since $u\in\cU_{ad}$, we have $\EE^\PP\left[\int_0^T |u|^{4r}\drm t\right] <\infty$. Therefore $\EE^\QQ \int_0^T |u|^4\drm t<\infty$.
	\qed

	\noindent We present the proof of the Lemma \ref{lem:M4}.
	
	\noindent\textbf{The Proof of Lemma \ref{lem:M4}:}
	For integers $N>|x|$, define $\tau_N:=\inf\{t\in[0,T]:|X(t)|>N\}\wedge T$.  Using the stopped state equation, the growth conditions in Assumption~\ref{Ass:1}, and $\mathbb E^\mathbb Q\int_0^T|u(s)|^4 \mathrm ds<\infty$, we obtain  $
	\EE^\QQ|X(t\wedge\tau_N)|^4\le C|x|^4+C\EE^\QQ\left|\int_0^{t\wedge\tau_N}b(s,X(s),u(s))\drm s\right|^4+C\EE^\QQ\left|\int_0^{t\wedge\tau_N}\sigma(s,X(s),u(s))\drm W(s)\right|^4
	+C\EE^\QQ\left|\int_0^{t\wedge\tau_N}\int_{\cE}\eta(s,X(s-),u(s),\zeta)\widetilde N(\drm\zeta,\drm s)\right|^4$, 
	where $C>0$ denotes a generic constant, independent of $t$ and $N$, which may change from line to line throughout the proof. We estimate the three integral terms one by one.
	
	For the drift term, using Hölder's inequality, $(a+b+c)^4\le27(a^4+b^4+c^4)$, and the linear growth condition of $b$ in Assumption \ref{Ass:1}, we obtain
	\begin{equation}\label{eq:b-est}
		\EE^\QQ\left|\int_0^{t\wedge\tau_N}b\drm s\right|^4
		\le C\left(1+\int_0^t\EE^\QQ|X(s\wedge\tau_N)|^4\drm s+\EE^\QQ\int_0^t|u(s)|^4\drm s\right).
	\end{equation}
	
	For the diffusion term, applying the B‑D‑G inequality, Hölder's inequality, and the linear growth condition of $\sigma$ in Assumption \ref{Ass:1}, we obtain
	\begin{equation}\label{eq:sigma-est}
		\EE^\QQ\left|\int_0^{t\wedge\tau_N}\sigma\drm W\right|^4
		\le C\left(1+\int_0^t\EE^\QQ|X(s\wedge\tau_N)|^4\drm s+\EE^\QQ\int_0^t|u(s)|^4\drm s\right).
	\end{equation}
	
	For the jump term, applying the Kunita inequality (see, e.g., \cite{applebaum2009levy}, Corollary 4.4.24) with $p=4$, together with the growth conditions  $\int_{\cE}|\eta(s,x,u,\zeta)|^2\nu(\drm\zeta)\le k_\eta(1+|x|^2+|u|^2)$, and $\int_{\cE}|\eta(s,x,u,\zeta)|^4\nu(\drm\zeta)\le k_\eta(1+|x|^4+|u|^4)$, 
	we obtain
	\begin{equation}\label{eq:eta-est}
		\EE^\QQ\left|\int_0^{t\wedge\tau_N}\int_{\cE}\eta \widetilde N(\drm\zeta,\drm s)\right|^4
		\le C\left(1+\int_0^t\EE^\QQ|X(s\wedge\tau_N)|^4\drm s+\EE^\QQ\int_0^t|u(s)|^4\drm s\right).
	\end{equation}
	
	Combining  \eqref{eq:b-est}, \eqref{eq:sigma-est} with \eqref{eq:eta-est} yields $\EE^\QQ|X(t\wedge\tau_N)|^4
	\le C\left(1+\EE^\QQ\int_0^T|u(s)|^4\drm s\right)+C\int_0^t\EE^\QQ|X(s\wedge\tau_N)|^4\drm s$, 
	where the fixed initial value $x$ is absorbed into  constant $C$. By using the  Gronwall's inequality, we have
	$\sup_{0\le t\le T}\EE^\QQ|X(t\wedge\tau_N)|^4\le C\left(1+\EE^\QQ\int_0^T|u(s)|^4\drm s\right)$,   which is 
	uniformly in $N$. Since $X$ has c\`adl\`ag paths on $[0,T]$, $\tau_N=T$ for all sufficiently large $N$, $\QQ$-almost surely. By applying the Fatou's lemma for each $t\in[0,T]$ and then taking the supremum over $t$, we conclude that $\sup_{0\le t\le T}\EE^\QQ|X(t)|^4\le C\left(1+\EE^\QQ\int_0^T|u(s)|^4\drm s\right)$.
	This proves \eqref{eq:X8}.
	\qed

	\noindent We proceed to introducing the proof of the Lemma \ref{lem:XL}.
	
	\noindent\textbf{The Proof of Lemma \ref{lem:XL}:} 
	Since Assumption~\ref{Ass:1} implies the uniform boundedness of $b_x,b_u,\sigma_x,\sigma_u$ and the $L^2(\cE,\nu)$- and $L^4(\cE,\nu)$-norms of $\eta_x,\eta_u$, applying localization, the BDG and Kunita inequalities, and
		Gronwall's inequality to the stopped equation, followed by
		Fatou's lemma, we obtain $
		\sup_{0\le t\le T}\EE^\QQ|X^l(t)|^4
		\le C\EE^\QQ\int_0^T|v(s)|^4\drm s<\infty.$

	Using the Cauchy--Schwarz inequality,
	Lemma~\ref{lem:moment-PQ} with $q=4$, and the estimate
	for $X^l$ obtained above, we have $$K:=\mathbb{E}^\mathbb{Q}\int_0^T |M(s)|^2|X^l(s)|^2\mathrm{d}s \le T \left(\sup_{0\le s\le T}\mathbb{E}^\mathbb{Q}|M(s)|^4\right)^{1/2} \left(\sup_{0\le s\le T}\mathbb{E}^\mathbb{Q}|X^l(s)|^4\right)^{1/2}
	<\infty.$$ Since $M$ and $X^l$ are càdlàg, replacing them
	by their left limits does not change this time integral.

	For $N\ge1$, let $\rho_N:=\inf\{t\in[0,T]:|M^l(t)|>N\}\wedge T$.
	Using the uniform bounds on $\lambda_k-1$ and
	$\lambda_k'$ implied by
	Assumptions~\ref{Ass:3} and \ref{Ass:4}, and applying
	the Itô isometry for the mutually orthogonal
	compensated counting martingales to
	\eqref{eq:varphi} stopped at $\rho_N$, we obtain $
	\mathbb{E}^\mathbb{Q}|M^l(t\wedge\rho_N)|^2
	\le C\int_0^t
	\mathbb{E}^\mathbb{Q}|M^l(s\wedge\rho_N)|^2\mathrm{d} s+CK$, 
	where $C>0$ is independent of $t$ and $N$. 
	
	Applying Gronwall's inequality yields $
	\sup_{0\le t\le T}
	\mathbb{E}^\mathbb{Q}|M^l(t\wedge\rho_N)|^2\le CK,$
	uniformly in $N$.
	Since $M^l$ has càdlàg paths on $[0,T]$,
	$\rho_N=T$ for all sufficiently large $N$,
	$\mathbb{Q}$-almost surely.
	Applying Fatou's lemma for each fixed $t$ and then
	taking the supremum over $t$, we conclude that $
	\sup_{0\le t\le T}\mathbb{E}^\mathbb{Q}|M^l(t)|^2<\infty.$

	\qed
	
	\noindent Now, we present the proof of Lemma \ref{lem:LXM}.
	
	\noindent\textbf{The proof of Lemma \ref{lem:LXM}.} We first prove \eqref{eq:X0}. Note that $\tilde{X}(0) = 0$, and $$\drm \tilde{X}(s)=b^1\drm s+\sigma^1\drm W(s)+\int_{\cE} \eta^1\tilde{N}(\drm \zeta,\drm s).$$  Here, $b^1 =b^\rho_x(s)\tilde{X}(s)+\left(b^\rho_x(s)-b_x(s)\right)X^l(s)+\left(b_u^\rho(s)-b_u(s)\right)v(s)$, $\sigma^1 =\sigma^\rho_x(s)\tilde{X}(s)+\left(\sigma^\rho_x(s)-\sigma_x(s)\right)X^l(s)+\left(\sigma_u^\rho(s)-\sigma_u(s)\right)v(s)$, and $\eta^1= \eta_x^\rho(s,\zeta)\tilde{X}(s-)+\left(\eta_x^\rho(s,\zeta)-\eta_x(s,\zeta)\right)X^l(s-)+\left[\eta_u^\rho(s,\zeta)-\eta_u(s,\zeta)\right]v(s)$, where $b_x^\rho$, $b_u^\rho$, $\sigma_x^\rho$, $\sigma_u^\rho$, $\eta_x^\rho$, and $\eta_u^\rho$ are defined in \eqref{eq:mth}.
	
	By   \eqref{eq:mth}, both the state and control derivatives are evaluated along the same interpolation path. Therefore, the expressions for $b^1$, $\sigma^1$, and $\eta^1$ above follow from the exact integral-form Taylor identities rather than from separate approximations. We repeatedly use that convergence in probability, together with an integrable dominating random variable independent of $\epsilon$,
		implies convergence in $L^1(\mathbb Q)$. Moreover, since $\xi^\epsilon$ is c\`adl\`ag, by applying Lemma~\ref{lem:X-X} and Fatou's lemma, we  imply $\sup_{0< s\le T}\mathbb E^{\mathbb Q}|\xi^\epsilon(s-)|^4\le C\epsilon^4$. Hence, $\xi^\epsilon(s-)\to0$ in probability for every $s\in[0,T]$. By the same argument, the moment estimates in Lemmas~\ref{lem:moment-PQ} and~\ref{lem:XL} also hold for the corresponding left limits. In particular, $\sup_{0<s\le T}\mathbb E^{\mathbb Q}|X^l(s-)|^4<\infty$, $\sup_{0<s\le T}\mathbb E^{\mathbb Q}|M(s-)|^4<\infty$,  and $ \sup_{0<s\le T}\mathbb E^{\mathbb Q}|M^l(s-)|^2<\infty$. 
	
Applying the BDG and Kunita inequalities yields $\mathbb E^{\mathbb Q}|\widetilde X(t)|^4 \le C\mathbb E^{\mathbb Q}\big[ \int_0^t (|b^1(s)|^4 + |\sigma^1(s)|^4) ds + \big(\int_0^t\int_{\mathcal E}|\eta^1(s,\zeta)|^2\nu(d\zeta) ds\big)^2 + \int_0^t\int_{\mathcal E}|\eta^1(s,\zeta)|^4\nu(d\zeta) ds \big]$, where $C>0$ hereafter denotes a generic constant independent of $(t,\epsilon)$ that may vary from line to line.
	
	We now consider the term $\EE^\QQ\int_0^t |b^1|^4\drm s$. By applying H\"{o}lder's inequality and Assumption~\ref{Ass:1}, we obtain that $\EE^\QQ\int_0^t |b^1|^4\drm s = \EE^\QQ \int_0^t |b^\rho_x \tilde{X} + (b^\rho_x -b_x)X^l + (b_u^\rho -b_u)v|^4\drm s \leq C \EE^\QQ \int_0^t |\tilde{X}|^4\drm s + C \EE^\QQ \int_0^t |b_x^\rho -b_x|^4 |X^l|^4\drm s + C \EE^\QQ \int_0^t |b_u^\rho -b_u|^4 |v|^4 \drm s$. For each $s\in[0,T]$, Assumption~\ref{Ass:1}, Lemma~\ref{lem:X-X}, and the boundedness of $v$ imply that $|b_x^\rho(s)-b_x(s)|^4\to0$ in probability. Since this quantity is uniformly bounded, the convergence in probability and domination argument gives $\mathcal K(\epsilon,s) := \mathbb E^{\mathbb Q} [|b_x^\rho(s)-b_x(s)|^4|X^l(s)|^4] \to0$. Moreover, $\mathcal K(\epsilon,s) \le C\mathbb E^{\mathbb Q}|X^l(s)|^4$ uniformly in $\epsilon$. Hence, the dominated convergence theorem with respect to $s$ yields $\int_0^T\mathcal K(\epsilon,s)\drm s\to0$.
	
	By using the same discussion on  $\EE^\QQ  \int_0^t \left\lvert b_u^\rho(s)-b_u(s)\right\rvert^4  \cdot\left\lvert v(s)\right\rvert^4 \drm s$ , along with the boundedness of $v$  in place  of the bound on $X^l$, we obtain $
	\EE^\QQ\int_0^T |b_u^\rho - b_u|^4 |v|^4\drm s \to 0 \text{ as }\epsilon\to 0. $ 
	Therefore, we have that 
	\begin{equation}\label{eq:bx} 
		\EE^\QQ\int_0^t \left\lvert b^1 \right\rvert^4\drm s \leq C  \EE^\QQ \int_0^t \lvert \tilde{X}(s)\rvert^4\drm s+ C \tilde{\cK}_1(\ep), 
	\end{equation}
	where $
	\tilde{\cK}_1(\ep)=\int_0^T \cK(\ep,s)\drm s+\EE^\QQ  \int_0^T \left\lvert b_u^\rho(s)-b_u(s)\right\rvert^4  \cdot\left\lvert v(s)\right\rvert^4 \drm s\to 0,\text{as } \ep\to 0.$

	Next, estimating $\EE^\QQ\int_0^t |\sigma^1|^4\drm s$ by using Young's and H\"{o}lder's inequalities yields $\EE^\QQ\int_0^t |\sigma^1|^4\drm s = \EE^\QQ\int_0^t |\sigma^\rho_x \tilde{X} + (\sigma^\rho_x -\sigma_x)X^l + (\sigma_u^\rho -\sigma_u)v|^4\drm s \leq C\EE^\QQ\int_0^t|\tilde{X}|^4\drm s + C\EE^\QQ\int_0^t |\sigma^\rho_x -\sigma_x|^4 |X^l|^4\drm s + C\EE^\QQ\int_0^t |\sigma_u^\rho -\sigma_u|^4 |v|^4\drm s$.  Following the same arguments used for $b^1$, we deduce that $\tilde{\cK}_2(\ep) := \EE^\QQ\int_0^T |\sigma^\rho_x(s)-\sigma_x(s)|^4 |X^l(s)|^4\drm s + \EE^\QQ\int_0^T |\sigma_u^\rho(s)-\sigma_u(s)|^4 |v(s)|^4\drm s \to 0$ as $\ep\to 0$.
	Thus,
	\begin{equation}\label{eq:s-1} 
		\EE^\QQ  \int_0^t |\sigma^1|^4\drm s  \le C\EE^\QQ\int_0^t \lvert\tilde{X}(s)\rvert^4\drm s+C \tilde{\cK}_2(\ep). 
	\end{equation}

	Similarly, estimating the corresponding jump diffusion term yields $\EE^\QQ\big(\int_0^t \int_{\cE} |\eta^1|^2\nu(\drm \zeta)\drm s\big)^2 \leq Ct\EE^\QQ \int_0^t |\tilde{X}(s-)|^4\drm s + C\EE^\QQ\big( \int_0^t \int_{\cE} |\eta^\rho_x(s,\zeta)-\eta_x(s,\zeta)|^2 |X^l(s-)|^2\nu(\drm \zeta)\drm s\big)^2 + C\EE^\QQ \big(\int_0^t \int_{\cE}|\eta_u^\rho(s,\zeta)-\eta_u(s,\zeta)|^2|v(s)|^2\nu(\drm \zeta)\drm s\big)^2$.
	
	Defining $\tilde{\cK}_3(\ep) := \EE^\QQ\big( \int_0^T \int_{\cE} |\eta^\rho_x(s,\zeta)-\eta_x(s,\zeta)|^2 |X^l(s-)|^2\nu(\drm \zeta)\drm s\big)^2 + \EE^\QQ \big(\int_0^T \int_{\cE} |\eta_u^\rho(s,\zeta)- \eta_u(s,\zeta)|^2 |v(s)|^2\nu(\drm \zeta)\drm s\big)^2$.
	
	For the first term, set $A_x^\epsilon(s) := \int_{\mathcal E} |\eta_x^\rho(s,\zeta)-\eta_x(s,\zeta)|^2 \nu(d\zeta)$.  By  Assumption~\ref{Ass:1}, applying Jensen's inequality with respect to $\rho$, together with $\xi^\epsilon(s-)\to0$ in probability and $\epsilon v(s)\to0$, we obtain that $A_x^\epsilon(s)\to0$ in probability.  By using the Cauchy--Schwarz inequality, $\mathbb E^{\mathbb Q} \big( \int_0^T A_x^\epsilon(s)|X^l(s-)|^2 ds \big)^2 \le T\int_0^T \mathbb E^{\mathbb Q} \big[ (A_x^\epsilon(s))^2|X^l(s-)|^4 \big]ds$. For each $s$, the integrand on the right converges to zero in probability and is dominated by $C|X^l(s-)|^4$. Since Lemma~\ref{lem:XL} gives $\sup_{0\le s\le T} \mathbb E^{\mathbb Q}|X^l(s)|^4<\infty$, the dominated-convergence argument described above, first with respect to $\omega$ and then with respect to $s$, yields $\mathbb E^{\mathbb Q} \big( \int_0^T A_x^\epsilon(s)|X^l(s-)|^2 ds \big)^2 \to0$. The term involving $\eta_u$ is treated in the same way with $A_u^\epsilon(s) := \int_{\mathcal E} |\eta_u^\rho(s,\zeta)-\eta_u(s,\zeta)|^2 \nu(d\zeta)$, using the boundedness of the perturbation $v$. Therefore, $\widetilde{\mathcal K}_3(\epsilon)\to0$.
	
	Finally, using $|a+b+c|^4\le C(|a|^4+|b|^4+|c|^4)$, we obtain $\EE^\QQ \int_0^t\int_{\cE} |\eta^1|^4\nu(\drm\zeta)\drm s \le C\EE^\QQ\int_0^t|\tilde X(s-)|^4\drm s + C\tilde{\cK}_4(\ep)$, where $\tilde{\cK}_4(\ep) := \EE^\QQ \int_0^T\int_{\cE} |\eta_x^\rho(s,\zeta)-\eta_x(s,\zeta)|^4 |X^l(s-)|^4 \nu(\drm\zeta)\drm s + \EE^\QQ \int_0^T\int_{\cE} |\eta_u^\rho(s,\zeta)-\eta_u(s,\zeta)|^4 |v(s)|^4 \nu(\drm\zeta)\drm s$. Based on Assumption~\ref{Ass:1},  by  using Lemma~\ref{lem:XL}, the boundedness of $v$, and the same dominated-convergence argument as above, we obtain $\tilde{\cK}_4(\ep)\to0$ as $\ep\to0$.

	Since $\tilde X$ is c\`adl\`ag, $\int_0^t\mathbb E^{\mathbb Q}|\tilde X(s-)|^4\drm s=\int_0^t\mathbb E^{\mathbb Q}|\tilde X(s)|^4\drm s$. Therefore, by combining the above estimates, we conclude that there exists a constant  \(C\) which is independent of time $t$ such that $ \EE^\QQ |\tilde{X}(t)|^4\leq C  \int_0^t \EE^\QQ  \lvert \tilde{X}(s)\rvert^4\drm s+C \sum_{i=1}^4 \tilde{\cK}_i(\ep).
	$ By applying Gronwall's inequality, we obtain that
	$
	\EE^\QQ |\tilde{X}(t)|^4\leq C e^{CT}\sum_{i=1}^4 \tilde{\cK}_i(\ep),
	$
	which implies that  $
	\sup_{0\leq t\leq T}\EE^\QQ |\tilde{X}(t)|^4\leq C e^{CT}\sum_{i=1}^4 \tilde{\cK}_i(\ep)\to 0 \text{ as }\ep \to 0.
	$

	We now prove the limit \eqref{eq:M0}. First, we observe that
	$$\drm \tilde{M}(s)=\sum_{k=1}^n  \Big\{(\lambda_k^\epsilon  -1)\tilde{M}  +(\lambda^\prime_{k,\rho}  -\lambda_k^\prime  )X^l  M  +\lambda_{k,\rho}^\prime  \tilde{X}  M  +(\lambda_k^\ep  -\lambda_k  )M^l  \Big\}\drm (Y_k(s)-s), $$ and $\tilde{M}(0)=0$. Here, $ \lambda'_{k,\rho}(s-) := \int_0^1 \lambda'_k \big( X(s-) + \rho(X^\epsilon(s-)-X(s-)), s- \big) \drm\rho $. For $k=1,\ldots,n$, define $F_k^\epsilon(s) := (\lambda_k^\epsilon(s-)-1)\widetilde M(s-) + (\lambda'_{k,\rho}(s-)-\lambda'_k(s-)) X^l(s-)M(s-) + \lambda'_{k,\rho}(s-) \widetilde X(s-)M(s-) + (\lambda_k^\epsilon(s-)-\lambda_k(s-)) M^l(s-)$. Since the compensated counting processes $Y_k(t)-t$, $k=1,\ldots,n$, are mutually orthogonal square-integrable martingales under $\mathbb Q$, applying It\^o's isometry first to the stopped equation and then removing the localization by the estimates below gives $\mathbb E^{\mathbb Q}|\widetilde M(t)|^2 = \mathbb E^{\mathbb Q} \int_0^t \sum_{k=1}^n|F_k^\epsilon(s)|^2 ds$.  Based on Assumptions \ref{Ass:3} and \ref{Ass:4}, we obtain $\EE^\QQ|\widetilde M(t)|^2 \le C\int_0^t \EE^\QQ|\widetilde M(s-)|^2\drm s + C\int_0^t \EE^\QQ\big[|\widetilde X(s-)|^2|M(s-)|^2\big]\drm s + C\sum_{k=1}^n \EE^\QQ \int_0^t |\lambda'_{k,\rho}(s-)-\lambda'_k(s-)|^2 |X^l(s-)|^2|M(s-)|^2\drm s + C\sum_{k=1}^n \EE^\QQ \int_0^t |\lambda_k^\epsilon(s-)-\lambda_k(s-)|^2 |M^l(s-)|^2\drm s$.
	
	For the term of $\EE^\QQ \int_0^t |\lambda_{k,\rho}^\prime(s-)-\lambda_k^\prime(s-)|^2 |X^l(s-)|^2 |M(s-)|^2\drm s$, based on the Assumptions \ref{Ass:3} and \ref{Ass:4}, by using Cauchy Schwarz twice, we can derive that $\EE^\QQ \int_0^t |\lambda_{k,\rho}^\prime(s-)-\lambda_k^\prime(s-)|^2 |X^l(s-)|^2 |M(s-)|^2\drm s   \leq C \int_0^t (\EE^\QQ |X^l(s-)|^4)^{\frac{1}{2}} (\EE^\QQ |M(s-)|^4)^{\frac{1}{2}}\drm s$. 
	Based on Assumptions~\ref{Ass:3} and \ref{Ass:4},  we then have $|\lambda'_{k,\rho}(s-)-\lambda'_k(s-)|^2\to 0$ in probability.   By the convergence-in-probability and domination argument described above, $\mathbb{E}[|\lambda'_{k,\rho}-\lambda'_k|^2|X^l|^2|M|^2]\to 0$ for each $s$. Since this quantity is bounded by $C(\mathbb{E}|X^l|^4)^{1/2}(\mathbb{E}|M|^4)^{1/2}<\infty$ uniformly in $s$, applying the  domination argument again with respect to $s$ gives
	$\mathbb{E}^{\mathbb{Q}}\int_0^t|\lambda'_{k,\rho}-\lambda'_k|^2|X^l|^2|M|^2\drm s \to 0$ as $\epsilon\to 0$.
	
	For the term $\EE^\QQ \int_0^t |\lambda^\ep_k(s-)-\lambda_k(s-)|^2 |M^l(s-)|^2\drm s$, also based on Assumptions~\ref{Ass:3} and \ref{Ass:4}, we have $|\lambda_k^\epsilon(s-)-\lambda_k(s-)|^2 \le C$ and $|\lambda_k^\epsilon(s-)-\lambda_k(s-)| \to 0$ in probability as $\epsilon \to 0$. Since $M^l$ does not depend on $\epsilon$ and $\sup_{0\le s\le T}\mathbb{E}^{\mathbb{Q}}|M^l(s)|^2<\infty$ by Lemma~\ref{lem:XL}, the integrand $|\lambda_k^\epsilon-\lambda_k|^2|M^l|^2$ is dominated by $C|M^l|^2$ with finite expectation. By the convergence-in-probability and domination argument described above, that is first in $\omega$, then in $s$, we have $\mathbb{E}^{\mathbb{Q}}\int_0^t|\lambda_k^\epsilon(s-)-\lambda_k(s-)|^2|M^l(s-)|^2\mathrm{d}s \to 0$ as $\epsilon \to 0$.
	
	Regarding the term $\EE^\QQ \int_0^t |\tilde{X}(s-)|^2 |M(s-)|^2\drm s$, by applying relevant inequalities and along with the convergence result established in the proof of \eqref{eq:X0}, we can show that $\EE^\QQ \int_0^t |\tilde{X}(s-)|^2 |M(s-)|^2\drm s = \int_0^t\EE^\QQ |\tilde{X}(s-)|^2 |M(s-)|^2\drm s \leq \int_0^t (\EE^\QQ |\tilde{X}(s-)|^4)^{\frac{1}{2}} (\EE^\QQ |M(s-)|^4)^{\frac{1}{2}}\drm s \leq \int_0^t(\sup_{0\le s \le T} \EE^\QQ |\tilde{X}(s)|^4)^{\frac{1}{2}}(\EE^\QQ |M(s-)|^4)^{\frac{1}{2}}\drm s \le [C e^{CT}\sum_{i=1}^4 \tilde{\cK}_i(\ep)]^{\frac{1}{2}}\int_0^t(\EE^\QQ |M(s-)|^4)^{\frac{1}{2}}\drm s \to 0$ as $\ep\to 0$.
	
	Taking $t=T$ in the preceding three estimates, we define $\widetilde{\mathcal K}_5(\epsilon) := \sum_{k=1}^n \EE^\QQ \int_0^T |\lambda'_{k,\rho}(s-)-\lambda'_k(s-)|^2 |X^l(s-)|^2 |M(s-)|^2\drm s + \sum_{k=1}^n \EE^\QQ \int_0^T |\lambda_k^\epsilon(s-)-\lambda_k(s-)|^2 |M^l(s-)|^2\drm s + \EE^\QQ \int_0^T |\widetilde X(s-)|^2 |M(s-)|^2\drm s$. Since $n<\infty$, the preceding estimates imply $\widetilde{\mathcal K}_5(\epsilon)\to0$ as  $\epsilon\to0$. 
	
	Therefore, $\EE^\QQ |\tilde{M}(t)|^2 \le C\EE^\QQ \int_0^t |\tilde{M}(s-)|^2\drm s + C \tilde{\cK}_5(\ep)$. Since $\widetilde M$ is c\`adl\`ag, its left-limit and current-value versions differ only at jump times. Hence, $\int_0^t\mathbb{E}^{\mathbb{Q}}|\widetilde M(s-)|^2\drm s = \int_0^t \mathbb{E}^{\mathbb{Q}}|\widetilde M(s)|^2\drm s$. By using Gronwall's inequality, we have $\EE^\QQ |\tilde{M}(t)|^2 \le C \tilde{\cK}_5(\ep)e^{CT}$, which implies that $\sup_{0\le t\le T}\EE^\QQ |\tilde{M}(t)|^2 \le C \tilde{\cK}_5(\ep)e^{CT}\to 0$ as $\ep\to 0$.
	
	\qed

	\noindent Now, we present the proof of Lemma \ref{lem:socal}.
	
	\noindent\textbf{The proof of Lemma \ref{lem:socal}.}  We first prove the solution class of  \eqref{eq:gamma}. Specifically, we will prove that 	\begin{align}\label{eq:socal-ga}
		\mathbb E^{\mathbb Q}[\sup_{0\leq t\leq T} |\bar\gamma(t)|^2+\int_0^T|\bar\beta(t)|^2\drm t +\int_0^T\int_{\cE}|\bar\theta(t,\zeta)|^2\nu(d\zeta)\drm t+\sum_{k=1}^n\int_0^T|\gamma^k(t)|^2\drm t]<\infty.
	\end{align} 
	Based on Assumption \ref{Ass:1} and Lemmas \ref{lem:moment-PQ} and \ref{lem:M4}, we have $\mathbb E^{\mathbb Q}[|\Phi(X(T))|^2+(\int_0^T|f(t)|\drm t)^2]<\infty$, and $\lambda_k-1$ is uniformly bounded. Hence, by applying the martingale representation property of the filtration generated by $W$, $\tilde N$, and $Y_1,\ldots,Y_n$, together with the standard theory of Lipschitz BSDEs with jumps, we obtain a unique solution to \eqref{eq:gamma} in the above class.
	
	We next prove the solution class of \eqref{eq:alpha}. Specifically, we will prove that 	\begin{align}\label{eq:socal-al}
		\mathbb E^{\mathbb Q}\bigg[& \sup_{0\leq t\leq T}|\alpha(t)|^{3/2}+\left(\int_0^T|\beta(t)|^2\mathrm dt\right)^{3/4}+\Big(\int_0^T\int_{\mathcal E}|\theta(t,\zeta)|^2\bar N(\mathrm d\zeta,\mathrm dt)\nonumber\\
		&+\sum_{k=1}^n\int_0^T|\xi^k(t)|^2\mathrm dY_k(t)\Big)^{3/4}\bigg]<\infty.
	\end{align}  
	Define
	\[
	h(t):=\sum_{k=1}^n\gamma^k(t)\lambda_k'(t-)M(t-)+M(t)f_x(t).
	\]
	From Assumptions~\ref{Ass:3} and \ref{Ass:4}, by applying Lemma~\ref{lem:moment-PQ} with moment order $6$, together with the square-integrability of $\gamma^k$, we obtain
	\[
	\mathbb E^{\mathbb Q}\int_0^T|M(t)\gamma^k(t)|^{3/2}\mathrm dt\leq\left(\mathbb E^{\mathbb Q}\int_0^T|M(t)|^6\mathrm dt\right)^{1/4}\times\left(\mathbb E^{\mathbb Q}\int_0^T|\gamma^k(t)|^2\mathrm dt\right)^{3/4}<\infty.
	\]
	Moreover, based on Assumption~\ref{Ass:1}, along Lemma~\ref{lem:moment-PQ} with  Lemma~\ref{lem:M4} , we have $\mathbb E^{\mathbb Q}\int_0^T|M(t)f_x(t)|^2\mathrm dt<\infty.$ Since $T<\infty$ and $n<\infty$, applying Hölder's inequality in time and the preceding estimates yield $
	\mathbb E^{\mathbb Q}\left(\int_0^T|h(t)|\mathrm dt\right)^{3/2}\leq C\sum_{k=1}^n\mathbb E^{\mathbb Q}\int_0^T|M(t)\gamma^k(t)|^{3/2}\mathrm dt+C\left(\mathbb E^{\mathbb Q}\int_0^T|M(t)f_x(t)|^2\mathrm dt\right)^{3/4}<\infty$. 
	Since the terminal condition $\Phi'(X(T))M(T)$ is square-integrable, then it belongs to $L^{3/2}(\mathbb Q)$. Therefore, 
	\[
	\mathbb E^{\mathbb Q}\left[|\Phi'(X(T))M(T)|^{3/2}+\left(\int_0^T|h(t)|\mathrm dt\right)^{3/2}\right]<\infty.
	\]
	
	It remains to estimate the Poisson-integrand term for the adjoint BSDE \eqref{eq:alpha}. Let $p=\frac32$, and $ p'=3$.  By interpolation between $L^2(\cE,\nu)$ and $L^4(\cE,\nu)$, we have $
	\|\eta_x(t,x,u,\cdot)\|_{L^3(\cE,\nu)}\leq\|\eta_x(t,x,u,\cdot)\|_{L^2(\cE,\nu)}^{1/3}\|\eta_x(t,x,u,\cdot)\|_{L^4(\cE,\nu)}^{2/3}<\infty$.  Hence, based on Assumption~\ref{Ass:1}, we have $
	\sup_{(t,x,u)}\left(\|\eta_x(t,x,u,\cdot)\|_{L^2(\cE,\nu)}+\|\eta_x(t,x,u,\cdot)\|_{L^3(\cE,\nu)}\right)<\infty$. 
	
	For predictable processes $\psi_1$ and $\psi_2$, define $$
	A_2(\psi_1):=\left[\mathbb E^{\mathbb Q}\left(\int_s^t\int_{\mathcal E}|\psi_1(r,\zeta)|^2\nu(\mathrm d\zeta)\mathrm dr\right)^{p/2}\right]^{1/p},$$ and $$ A_p(\psi_2):=\left[\mathbb E^{\mathbb Q}\int_s^t\int_{\mathcal E}|\psi_2(r,\zeta)|^p\nu(\mathrm d\zeta)\mathrm dr\right]^{1/p}.$$ 
	Let $L_u:=\int_s^u\int_{\mathcal E}\psi(r,\zeta)\widetilde N(\mathrm d\zeta,\mathrm dr)$ for any $s\leq u\leq t.$  Then, we have $
	[L]_t=\int_s^t\int_{\mathcal E}|\psi(r,\zeta)|^2\bar N(\mathrm d\zeta,\mathrm dr)$. 
	By the Burkholder--Davis--Gundy inequality, we have $
	c_p  \mathbb E^\mathbb Q \left([L]_t^{p/2}\right) \le \mathbb E^\mathbb Q \left[ \sup_{s\le u\le t}|L_u|^p \right] \le C_p  \mathbb E^\mathbb Q \left([L]_t^{p/2}\right).
	$ Thus, $
	\left\|\sup_{s\le u\le t}|L_u|\right\|_{L^p(\mathbb Q)} \asymp \left\{\mathbb E^\mathbb Q \left([L]_t^{p/2}\right)\right\}^{1/p}.$ Here $\asymp$ denotes two-sided estimates up to positive constants independent of $\psi$.
	Combining with the definition of $	\|\psi\|_{\mathbb J^p(s,t)}$,  we  obtain
	\[\|\psi\|_{\mathbb J^p(s,t)} \asymp \left\| \sup_{s\le u\le t} \left| \int_s^u\int_{\mathcal E} \psi d\widetilde N \right| \right\|_{L^p(\mathbb Q)}. 
	\]
	
	On the other hand, applying Theorem 1.1 given in \cite{dirksen2014ito} in the scalar case with $q=2$ gives
	\[
	\left\|\sup_{s\leq u\leq t}\left|\int_s^u\int_{\mathcal E}\psi(r,\zeta)\widetilde N(\mathrm d\zeta,\mathrm dr)\right|\right\|_{L^p(\mathbb Q)}\asymp\inf_{\psi=\psi_1+\psi_2}\{A_2(\psi_1)+A_p(\psi_2)\}.
	\]
	Indeed, since $p=3/2<2=q$, the relevant space in Dirksen's notation is $(\mathcal S_2^p\cap\mathcal D_{2,2}^p)+\mathcal D_{p,2}^p$. In the one-dimensional case, $\mathcal S_2^p=\mathcal D_{2,2}^p$, both corresponding to $A_2$, while $\mathcal D_{p,2}^p$ corresponds to $A_p$. Consequently,
	\[
	\|\psi\|_{\mathbb J^p(s,t)}\asymp\inf_{\psi=\psi_1+\psi_2}\{A_2(\psi_1)+A_p(\psi_2)\},
	\]
	where the infimum is taken over all measurable
	decompositions $\psi=\psi_1+\psi_2$ for which
	$A_2(\psi_1)+A_p(\psi_2)<\infty$.  Let $\psi=\psi_1+\psi_2$ be any such decomposition.
	
	For the part of  $\psi_1$, based on Assumption \ref{Ass:1},   together with the Cauchy--Schwarz inequality first in $\zeta$ and then in time, we have $
	\left\|\int_s^t\left|\int_{\mathcal E}\eta_x(r,\zeta)\psi_1(r,\zeta)\nu(\mathrm d\zeta)\right|\mathrm dr\right\|_{L^p(\mathbb Q)}\leq C\left\|\int_s^t\left(\int_{\mathcal E}|\psi_1(r,\zeta)|^2\nu(\mathrm d\zeta)\right)^{1/2}\mathrm dr\right\|_{L^p(\mathbb Q)}\leq C(t-s)^{1/2}A_2(\psi_1).$
	
	For the part of $\psi_2$, since $p'=3$, using the uniform $L^3(\cE,\nu)$ bound on $\eta_x$ and Hölder's inequality in $\zeta$ yield $
	\left|\int_{\mathcal E}\eta_x(r,\zeta)\psi_2(r,\zeta)\nu(\mathrm d\zeta)\right|\leq C\left(\int_{\mathcal E}|\psi_2(r,\zeta)|^p\nu(\mathrm d\zeta)\right)^{1/p}$. Then, by
	applying Hölder's inequality in time, we obtain $
	\left\|\int_s^t\left|\int_{\mathcal E}\eta_x(r,\zeta)\psi_2(r,\zeta)\nu(\mathrm d\zeta)\right|\mathrm dr\right\|_{L^p(\mathbb Q)}\leq C(t-s)^{1-1/p}\left[\mathbb E^{\mathbb Q}\int_s^t\int_{\mathcal E}|\psi_2(r,\zeta)|^p\nu(\mathrm d\zeta)\mathrm dr\right]^{1/p}=C(t-s)^{1/3}A_p(\psi_2)$.
	
	Therefore, for every such jointly measurable decomposition $\psi=\psi_1+\psi_2$, we obtain $ \left\|\int_s^t\left|\int_{\mathcal E}\eta_x(r,\zeta)\psi(r,\zeta)\nu(\mathrm d\zeta)\right|\mathrm dr\right\|_{L^p(\mathbb Q)}\leq C\left((t-s)^{1/2}+(t-s)^{1/3}\right)\left(A_2(\psi_1)+A_p(\psi_2)\right)$. 
	Taking the infimum over all such decompositions and using the preceding norm equivalence yields
	\[
	\left\|\int_s^t\left|\int_{\mathcal E}\eta_x(r,\zeta)\psi(r,\zeta)\nu(\mathrm d\zeta)\right|\mathrm dr\right\|_{L^p(\mathbb Q)}\leq C\left((t-s)^{1/2}+(t-s)^{1/3}\right)\|\psi\|_{\mathbb J^p(s,t)}.
	\]
	Since simple predictable integrands are dense in $\mathbb J^p(s,t)$, the preceding estimate extends to every $\psi\in\mathbb J^p(s,t)$.
	
	We now apply this estimate to the generator of \eqref{eq:alpha}. Fix a terminal value
	$\chi\in L^p(\mathbb Q)$. For each input $(\alpha ,\beta , \theta ) \in\mathbb S^p(s,t)\times\mathbb H^p(s,t) \times\mathbb J^p(s,t)$,
	freeze the generator of \eqref{eq:alpha} at this input and impose the terminal condition $\alpha^{\mathrm{out}}(t)=\chi$. Then, by using the  martingale representation, we obtain an output
	$(\alpha^{\mathrm{out}},\beta^{\mathrm{out}}, \theta^{\mathrm{out}},\xi^{\mathrm{out}})$. Combining the preceding integrability estimates,  with the 
	Doob's inequality and B-D-G inequality, we derive that this output belongs to
	$\mathbb S^p(s,t)\times\mathbb H^p(s,t) \times\mathbb J^p(s,t)\times\mathbb K^p(s,t)$. Define
	$ \mathcal T (\alpha ,\beta , \theta ) := (\alpha^{\mathrm{out}},\beta^{\mathrm{out}}, \theta^{\mathrm{out}}).$

	Let $\Delta G$ denote the difference of the generators corresponding to two inputs $(\alpha,\beta,\theta)$ and $(\alpha',\beta',\theta')$. By the boundedness of $b_x$ and $\sigma_x$,  by using the Cauchy--Schwarz inequality in time, and the preceding estimate for the jump term, we obtain $
	\left\|\int_s^t|\Delta G(r)|\mathrm dr\right\|_{L^p(\mathbb Q)}\leq C\Big[(t-s)\|\Delta\alpha\|_{\mathbb S^p(s,t)}+(t-s)^{1/2}\|\Delta\beta\|_{\mathbb H^p(s,t)}+\left((t-s)^{1/2}+(t-s)^{1/3}\right)\|\Delta\theta\|_{\mathbb J^p(s,t)}\Big]$.

	Let $
	Z:=\int_s^t\Delta G(r)\mathrm dr$, and $ \mathcal M_r:=\mathbb E^{\mathbb Q}\left[Z\mid\widehat{\mathcal F}_r\right]$, for any $ s\leq r\leq t$. 
	Since the terminal conditions of the two linear BSDEs are same, then $
	\Delta\alpha_r^{ \mathrm{out}}=\mathcal M_r-\int_s^r\Delta G(v)\mathrm dv=\mathbb E^{\mathbb Q}\left[\left.\int_r^t\Delta G(v)\mathrm dv\right|\widehat{\mathcal F}_r\right].$ 
	In particular, $
	|\Delta\alpha_r^{ \mathrm{out}}|\leq\mathbb E^{\mathbb Q}\left[\left.\int_s^t|\Delta G(v)|\mathrm dv\right|\widehat{\mathcal F}_r\right]$. 
	Hence, applying Doob's inequality, we have 
	\[
	\|\Delta\alpha^{ \mathrm{out}}\|_{\mathbb S^p(s,t)}\leq C_p\left\|\int_s^t|\Delta G(r)|\mathrm dr\right\|_{L^p(\mathbb Q)}.
	\]
	Moreover, using the martingale representation theorem to the martingale $\mathcal M_r-\mathcal M_s$ yields $
	\mathcal M_r-\mathcal M_s=\int_s^r\Delta\beta^{ \mathrm{out}}(v)\mathrm dW(v)+\int_s^r\int_{\mathcal E}\Delta\theta^{ \mathrm{out}}(v,\zeta)\widetilde N(\mathrm d\zeta,\mathrm dv)+\sum_{k=1}^n\int_s^r\Delta\xi^{k,\mathrm{out}}(v)\mathrm d(Y_k(v)-v)$. 
	By the orthogonality of the driving martingales established above, we have  $
	[\mathcal M]_t-[\mathcal M]_s=\int_s^t|\Delta\beta^{ \mathrm{out}}(r)|^2\mathrm dr+\int_s^t\int_{\mathcal E}|\Delta\theta^{ \mathrm{out}}(r,\zeta)|^2\bar N(\mathrm d\zeta,\mathrm dr)+\sum_{k=1}^n\int_s^t|\Delta\xi^{k,\mathrm{out}}(r)|^2\mathrm dY_k(r)$. 
	Then, applying the B-D-G and Doob inequalities yields $
	\left\{\mathbb E^{\mathbb Q}\left[\left([\mathcal M]_t-[\mathcal M]_s\right)^{p/2}\right]\right\}^{1/p}\leq C_p\left\|\sup_{s\leq r\leq t}|\mathcal M_r-\mathcal M_s|\right\|_{L^p(\mathbb Q)}\leq C_p\|Z\|_{L^p(\mathbb Q)}\leq C_p\left\|\int_s^t|\Delta G(r)|\mathrm dr\right\|_{L^p(\mathbb Q)}.$
	
	Therefore,
	\[
	\|\Delta\beta^{ \mathrm{out}}\|_{\mathbb H^p(s,t)}+\|\Delta\theta^{ \mathrm{out}}\|_{\mathbb J^p(s,t)}+\|\Delta\xi^{ \mathrm{out}}\|_{\mathbb K^p(s,t)}\leq C_p\left\|\int_s^t|\Delta G(r)|\mathrm dr\right\|_{L^p(\mathbb Q)}.
	\]
	Together with the estimate for $\Delta\alpha^{ \mathrm{out}}$, this gives $
	\|\Delta\alpha^{ \mathrm{out}}\|_{\mathbb S^p(s,t)}+\|\Delta\beta^{ \mathrm{out}}\|_{\mathbb H^p(s,t)}+\|\Delta\theta^{ \mathrm{out}}\|_{\mathbb J^p(s,t)}+\|\Delta\xi^{ \mathrm{out}}\|_{\mathbb K^p(s,t)}\leq C\left\|\int_s^t|\Delta G(r)|\mathrm dr\right\|_{L^p(\mathbb Q)}\leq C\left((t-s)+(t-s)^{1/2}+(t-s)^{1/3}\right)\Big[\|\Delta\alpha  \|_{\mathbb S^p(s,t)}+\|\Delta\beta  \|_{\mathbb H^p(s,t)}+\|\Delta\theta  \|_{\mathbb J^p(s,t)}\Big].$  
	Combined with the preceding difference estimate, this
	shows that $\mathcal T$ maps the product space into itself
	and that every corresponding $\xi^{\mathrm{out}}$
	belongs to $\mathbb K^p(s,t)$.
	Since $(t-s)+(t-s)^{1/2}+(t-s)^{1/3}\to0$ as $t-s\to0$, $\mathcal T$ is contractive  on sufficiently short intervals and admits a unique fixed point $(\alpha,\beta,\theta)$.
	The process $\xi$ does not appear in the generator. It is fixed uniquely via martingale representation, and satisfies $\|\xi\|_{\mathbb K^p(s,t)}<\infty$.
	
	Choose a deterministic partition $0=t_0<t_1<\cdots<t_m=T$, with $\Delta t = t_{i+1}-t_i$ for $i=0,\dots,m-1$, where $\Delta t$ is sufficiently small.
	We first solve the equation on $[t_{m-1},T]$ with terminal condition
	$\Phi'(X(T))M(T)$. We then proceed backward, using
	$\alpha(t_i)$ obtained on $[t_i,t_{i+1}]$ as the terminal condition
	on $[t_{i-1},t_i]$. Since $\alpha\in\mathbb S^p$ on each local
	interval, then every such terminal value belongs to
	$L^p(\widehat{\mathcal F}_{t_i},\mathbb Q)$. Since the partition is finite, the local estimates imply that the patched solution belongs to $\mathbb S^p\times\mathbb H^p\times\mathbb J^p\times\mathbb K^p$. And global uniqueness follows by applying the local uniqueness result successively backward over the same partition.
\qed

\section{Implementation Details}\label{App:2}
\paragraph{CE and full-information benchmarks.}
The CE benchmark uses the feedback coefficients from the
full-information LQ problem. Let $P_{\mathrm{FI}}$ and $Q_{\mathrm{FI}}$ satisfy
\[
\begin{aligned}
	P_{\mathrm{FI}}'(t)
	&=2\vartheta P_{\mathrm{FI}}(t)
	+\frac{\rho^2}{2k_2}P_{\mathrm{FI}}(t)^2,
	&\qquad P_{\mathrm{FI}}(T)&=2,\\
	Q_{\mathrm{FI}}'(t)
	&=\vartheta Q_{\mathrm{FI}}(t)
	-\vartheta\bar X P_{\mathrm{FI}}(t)
	+\frac{\rho^2}{2k_2}
	P_{\mathrm{FI}}(t)Q_{\mathrm{FI}}(t),
	&\qquad Q_{\mathrm{FI}}(T)&=-k_1.
\end{aligned}
\]
At time $t_i$, the CE control is given by $u_i^{\mathrm{CE}}=\frac{\rho}{2k_2}\left(
P_{\mathrm{FI}}(t_i)\widehat{\pi}_i+Q_{\mathrm{FI}}(t_i)\right)$,
where $\widehat{\pi}_i$ is the particle-filter mean based on the
observations available at time $t_i$. The full-information benchmark
uses the same feedback law, with $\widehat{\pi}_i$ replaced by $X_i$.

\paragraph{State simulation and filtering.}

Let $X_i$ denote the simulated state at $t_i$. We employ the Euler scheme
\[
X_{i+1}=X_i+\bigl[\vartheta(\bar X-X_i)-\rho u_i\bigr]\Delta+\sigma\sqrt{\Delta}Z_i+J_i,
\]
where $Z_i$ is standard normal and $J_i$ is the compound-Poisson increment with intensity $\lambda_J$. The two jump sizes are $\eta_+=\sqrt{\frac{V_J(1-p_{\rm up})}{p_{\rm up}}}$, and $
\eta_-=-\sqrt{\frac{V_Jp_{\rm up}}{1-p_{\rm up}}}$,  with probabilities $p_{\rm up}$ and $1-p_{\rm up}$, respectively. This choice yields a zero jump mean.  

 The observation intensities are held fixed at the left endpoint of each time interval. For the intensities used here, $\sum_{k=1}^n\lambda_k(x)=\bar a$, for $x\in\mathbb R$, 
 so that the total observation intensity is constant. We write $\Delta Y_i^k
 :=Y_{t_{i+1}}^k-Y_{t_i}^k$ for any $ k=1,\ldots,n$.  At time $t_i$, the control $u_i$ is computed from the particle filter based on the observations available up to $t_i$, and is therefore determined before the new increment $\Delta Y_i=(\Delta Y_i^1,\ldots,\Delta Y_i^n)$ is observed. Under $\mathbb P$, the total count over $[t_i,t_{i+1})$ is sampled as $
 N_i^{\mathrm{tot}} \sim \operatorname{Poisson}(\bar a\Delta)$, 
 and each event is assigned mark $k$ with probability
 $\lambda_k(X_i)/\bar a$. The particle weights are updated according to
 \begin{equation}\label{eq:pf-weight-app}
 	\log w_i^j \leftarrow \log w_i^j 	+ \sum_{k=1}^n \Delta Y_i^k\log\lambda_k(x_i^j).
 \end{equation}
 The compensator contribution is $-\Delta\sum_{k=1}^n\lambda_k(x_i^j)=-\bar a\Delta$, 
 which is common to all particles and therefore cancels upon
 normalization. Systematic resampling is applied when the effective
 sample size falls below half the particle count. The particles are
 then propagated with the applied control, yielding the filtered mean
 $\widehat\pi_{i+1}$ and variance $\widehat V_{i+1}$.

\paragraph{Backward projection.}
Define $x_i^s=\frac{X_i-\bar X}{0.5}$, $ m_i=\frac{\widehat\pi_i-\bar X}{0.5}$, and $ 
s_i=\frac{\sqrt{\widehat V_i}}{0.3}$.  The projection basis for the full filtration is $
\psi_i=	\bigl(1,x_i^s,(x_i^s)^2,m_i,m_i^2,s_i,s_i^2,x_i^sm_i,x_i^ss_i,m_is_i\bigr)$, 
while the observation-adapted basis is $\phi_i	=(1,m_i,m_i^2,s_i,s_i^2,m_is_i)$. For a fitted coefficient vector $c=(c_0,\ldots,c_9)^\top$, taking the conditional expectation of
the quadratic basis with respect to the weighted particle approximation gives
\begin{equation}\label{eq:projection-coeff-app}
	\Pi(c)=\bigl(c_0, c_1+c_3, c_2+c_4+c_7, c_5, 0.36c_2+c_6, c_8+c_9\bigr)^\top.
\end{equation}
The design states are generated under $\mathbb P$, whereas the
conditional branches used in the backward recursion are generated
under $\mathbb Q$. Each conditional branch is initialized from the
current latent state and the current weighted particle cloud. The
filter is then updated using that branch's count increment and
propagated under the applied control.
 
Since the reference intensities are equal to one, the reference
count increments are independent Poisson random variables with mean
$\Delta$. The observation-martingale coefficient is estimated by
$\gamma_i^k\approx\frac{1}{B}\sum_{b=1}^{B}\left(G_{i+1}^{(+k,b)}-G_{i+1}^{(0,b)}\right)$,
where the paired branches share all randomness except that the
$+k$ branch carries one additional count in channel $k$.

The weighted conditional expectation
$\widehat{\mathbb E}_i^{\mathbb Q}\left[L_{i+1}a_{i+1}\right]$ in \eqref{eq:num-a} is evaluated by Poisson tilting. For every observation channel $k$, the reference mean $\Delta$ is replaced by $\lambda_k(X_i)\Delta$.
Equivalently, under the tilted one-step law, the channel counts are
conditionally independent Poisson random variables with means
$\lambda_k(X_i)\Delta$, $k=1,\ldots,n$. This change of measure is
exact for the adopted frozen-intensity one-step scheme.
 The targets in \eqref{eq:num-G}--\eqref{eq:num-a} are obtained via ridge projection with an unpenalized intercept. Equation~\eqref{eq:projection-coeff-app} gives the observation-adapted projection of the fitted first adjoint, from which the stationarity target is obtained through
\eqref{eq:stationarity-map}. Three backward passes are performed with independent branch simulations for each fixed-point update.

\paragraph{Stopping rule and diagnostics.}
The iteration is monitored by the validation residual $r_{\rm stat}$
defined in \eqref{eq:fp-residual}. The same validation random numbers
are retained across update attempts. If a candidate residual exceeds
the previous accepted residual by more than $3\%$ in relative terms,
the candidate is rejected, the previous policy is restored, and the
damping parameter is halved, subject to a minimum value of $0.0625$.
Tolerance-based stopping is allowed only after at least one
stationarity-map correction has been accepted.  The iteration terminates when $r_{\mathrm{stat}}\leq0.5\%$, when the relative update size falls below $0.25\%$, or after
six accepted updates. At most four rollbacks are allowed.

The adjoint diagnostics are evaluated on independently generated
holdout states using fresh conditional branches under $\mathbb Q$. 
The reported adjoint diagnostics are pooled relative RMS deviations from
\eqref{eq:num-a} and \eqref{eq:num-G}, normalized by the combined
root-mean-square magnitudes of the left- and right-hand sides of the
respective recursions. The reported coupling diagnostic is the RMS of $
\sum_k\gamma_i^k\lambda_k'(X_i)$  over the same holdout calculations. These quantities assess consistency with the adopted discrete numerical scheme and are not error estimates
for the continuous-time adjoint equations. The numerical control bound
is $8$ and remains inactive in all reported experiments.

\section{Numerical-Resolution Sensitivity}\label{App:3}
We repeat the nominal experiment at $\beta_p=15$ while varying one numerical
resolution parameter at a time. The time step is varied over
$\Delta\in\{0.04,0.02,0.01\}$ with $N_p=128$ particles, and the particle count is
varied over $N_p\in\{64,128,256\}$ with $\Delta=0.02$. All remaining path counts,
branch counts, regression bases, stopping rules, and random seeds are unchanged.

\begin{table}[htbp]
	\centering
	\scriptsize
	\setlength{\tabcolsep}{3.2pt}
	\caption{Numerical-resolution sensitivity at $\beta_p=15$. Only one resolution parameter is varied at a time. The cost difference and its paired 95\% confidence interval are multiplied by $10^5$.}
	\label{tab:numerical-stability}
	\begin{tabular}{cccccc}
		\toprule
		$\Delta$ & $N_p$ & $G_u$ (\%) & $10^5\Delta J$ [paired 95\% CI] & Coupling RMS & Holdout $r_{\rm stat}$ (\%)\\
		\midrule
		\multicolumn{6}{l}{Time-step sensitivity ($N_p=128$)}\\
		0.04 & 128 & 0.217 & -1.355 [-2.862,0.151] & 0.0456 & 0.683 \\
		\textbf{0.02} & \textbf{128} & \textbf{0.698} & \textbf{-0.282 [-2.504,1.940]} & \textbf{0.0475} & \textbf{0.883} \\
		0.01 & 128 & 0.775 & 0.369 [-1.510,2.248] & 0.0420 & 0.866 \\
		\addlinespace
		\multicolumn{6}{l}{Particle-count sensitivity ($\Delta=0.02$)}\\
		0.02 & 64 & 0.173 & -0.402 [-1.877,1.074] & 0.0402 & 0.525 \\
		\textbf{0.02} & \textbf{128} & \textbf{0.698} & \textbf{-0.282 [-2.504,1.940]} & \textbf{0.0475} & \textbf{0.883} \\
		0.02 & 256 & 0.824 & 0.814 [-0.844,2.471] & 0.0429 & 0.869 \\
		\bottomrule
	\end{tabular}
\end{table}

The purpose of this comparison is to assess the sensitivity of the
reported numerical conclusions to the selected time step and particle
count, rather than to estimate a convergence rate.
Across the five distinct numerical configurations, the estimated
coupling RMS ranges from $0.0402$ to $0.0475$, the relative policy-map
gap from $0.173\%$ to $0.824\%$, and the holdout stationarity residual
from $0.525\%$ to $0.883\%$. The paired $95\%$ confidence interval for
the SMP--CE cost difference contains zero in every configuration, and
the numerical control bound is inactive throughout. Hence, the main
qualitative findings at the nominal resolution are unchanged over the
tested time steps and particle counts.

\end{appendices}
 
\end{document}